\documentclass[a4paper, 11pt]{article}
\usepackage[T1]{fontenc}
\usepackage[utf8]{inputenc}
\usepackage[english]{babel}
\usepackage{amsmath}
\usepackage{amssymb}
\usepackage{amsfonts}
\usepackage{amsthm}
\usepackage{enumitem}
\usepackage{mathtools}
\usepackage{csquotes}
\usepackage{hyperref}

\numberwithin{equation}{section}
\newtheorem{theorem}{Theorem}[section]
\newtheorem{lemma}[theorem]{Lemma}

\newtheorem{proposition}[theorem]{Proposition}
\theoremstyle{definition}

\newtheorem{remark}[theorem]{Remark}

\newcommand{\per}{\mathrm{per}}
\newcommand{\e}{\mathrm{e}}
\newcommand{\R}{\mathbb R}
\newcommand{\C}{\mathbb C}

\newcommand{\N}{\mathbb N}

\newcommand{\Om}{\mathcal O}
\newcommand{\SL}{\mathcal S}
\newcommand{\LL}{\mathcal L}
\newcommand{\T}{\mathcal T}
\newcommand{\M}{\mathcal M}
\newcommand{\K}{\mathcal K}
\newcommand{\Pro}{\mathcal P}
\newcommand{\A}{\mathcal A}
\newcommand{\U}{\mathcal U}
\newcommand{\Ch}{\mathcal C_h^L}

\DeclareMathOperator{\im}{im}

\allowdisplaybreaks

\begin{document}
	
\title{Global bifurcation of capillary-gravity water waves with overhanging profiles and arbitrary vorticity}
\date{}
\author{Erik Wahlén\footnotemark[1] \and Jörg Weber\footnotemark[2]}
\renewcommand{\thefootnote}{\fnsymbol{footnote}}
\footnotetext[1]{Centre for Mathematical Sciences, Lund University, 221 00 Lund, Sweden; \tt\href{mailto:erik.wahlen@math.lu.se}{erik.wahlen@math.lu.se}}
\footnotetext[2]{Centre for Mathematical Sciences, Lund University, 221 00 Lund, Sweden; \tt\href{mailto:jorg.weber@math.lu.se}{jorg.weber@math.lu.se}}

\maketitle

\begin{abstract}
	We study two-dimensional periodic capillary-gravity water waves propagating at the free surface of water in a flow with arbitrary, prescribed vorticity over a flat bed. Using conformal mappings and a new Babenko-type reformulation of Bernoulli's equation, the problem is equivalently cast into the form ``identity plus compact'', which is amenable to Rabinowitz' global bifurcation theorem, while no restrictions on the geometry of the surface profile and no assumptions regarding the absence of stagnation points in the flow have to be made. Within the scope of this new formulation, local and global solution curves, bifurcating from laminar flows with a flat surface, are constructed.
\end{abstract}

\section{Introduction}
In this paper we construct two-dimensional periodic steady water waves of large amplitude propagating in an inviscid, incompressible, and homogeneous fluid with a flat bed. The waves we consider are capillary-gravity waves, that is, they are under the influence of both gravity and surface tension. While many papers have been devoted to the construction of water waves of small/large amplitude by means of local/global bifurcation arguments, the novelty of this paper is that our approach is very general and accounts for \textit{all} of the following features:

First, our waves are rotational. The vorticity function can be arbitrary, as long as a very mild regularity assumption (to be specified later) is satisfied. While the mathematical study of certain steady water-wave problems has an extensive history, most of the literature of the previous century dealt with irrotational flows, which enjoy the advantage of being thoroughly treatable by tools of complex analysis, since the water velocity can be written as the gradient of a harmonic potential. For surveys on irrotational flows, we refer to \cite{Groves04,Toland96}. However, there are many situations in which it is necessary to take vorticity into account, for example, in the presence of underlying non-uniform currents.

Second, our waves can have interior stagnation points and arbitrarily many critical layers. In particular, we do not assume that the stream function is strictly monotone with respect to the vertical coordinate. It was this assumption, giving rise to the semi-hodograph transformation of Dubreil-Jacotin \cite{Dubreil34}, that Constantin and Strauss utilised to construct large steady periodic water waves with vorticity for the first time in their breakthrough paper \cite{ConsStr04}, which neglects surface tension. As for taking into account capillary effects and/or stratification, we mention \cite{Wahlen06b,Wahlen06a} and \cite{EKLM20,Walsh09,Walsh14a,Walsh14b}, respectively. However, in the presence of stagnation points, such a transformation is no longer justified and the theory breaks down.

Third, our waves can have overturning profiles. Such waves are known to exist in the irrotational context with surface tension \cite{AkersAmbroseWright14, Crapper57, deBoeck14b, Kinnersley76} and in the case of constant vorticity without surface tension \cite{HurWheeler20, HurWheeler21}, and can therefore also be expected in the present setting.
 In order to overcome the inherited downsides of the above-mentioned semi-hodograph transformation and to allow for stagnation points and critical layers, in many papers a naive flattening transform, where on each vertical ray the vertical coordinate is scaled to a constant, was used. We like to mention \cite{AasVar18,EhrnEschWahl11,EhrnVill08,EhrnWahl15,HenMat13,KozKuz14,Varholm20,Wahlen09} at this point. The immediate disadvantage of this naive flattening is that for this the surface profile needs to be a graph, thus precluding overhanging waves. We therefore make use of a flattening via a conformal mapping. 
This is classical in the irrotational context and can  be used to reformulate the problem either as a singular integral equation for the tangent angle of the free surface \cite{Groves04, Toland96} or as a nonlinear pseudodifferential equation for the surface elevation in the new variables introduced by Babenko \cite{Babenko87b} (see also \cite{Groves04} and reference therein). The latter approach was extended to constant vorticity by Constantin and V\u{a}rv\u{a}ruc\u{a} in \cite{ConsVarva11} and then utilised, together with Strauss, in \cite{ConsStrVarv16} to obtain pure gravity water waves of large amplitude allowing for overhanging surface profiles. 
We also like to mention \cite{Martin13}, where, in the same spirit, local bifurcation for capillary-gravity waves with constant vorticity was investigated.

Fourth, we construct steady water waves of large amplitude. In order to do this, the usual strategy is to apply a global bifurcation theorem. First, degree theoretic methods using the Healey--Simpson degree and Kielhöfer degree, respectively, were used in \cite{ConsStr04} and \cite{Walsh14b}, respectively, albeit relying on the semi-hodograph transform in order to prove admissibility of the nonlinear operator. Second, analytic global bifurcation, going back to Dancer \cite{Dancer73} and Buffoni and Toland \cite{BuffTol03}, was, for example, used in \cite{ConsStrVarv16,Varholm20} to construct a global continuum of solutions, and has the advantage of producing a curve of solutions admitting locally an analytic reparametrisation. However, to this end the occurring nonlinear operators have to be analytic; this, in turn, requires that the vorticity function be analytic unless the semi-hodograph transform has been applied in the first place. Third, if one wants to relax this assumption on analyticity and to apply the global bifurcation theorem of Rabinowitz \cite{Rabinowitz71,Kielhoefer}, the operator equation needs to have the form \enquote{identity plus compact}. Therefore, the water wave problem has to be reformulated suitably, and at exactly this point the presence of surface tension is crucial. In this spirit, the strategy of \cite{HenMat14,Matioc14} was to invert the curvature operator in order to \enquote{gain a derivative}; however, restricted to the case where the surface profile is a graph, and where the vorticity function satisfies a monotonicity assumption, thus giving rise to a reformulation via reduction to the boundary. We also mention \cite{AmbroseStraussWright16} where a similar strategy is used for vortex sheets in the irrotational context, allowing for overhanging waves and using a reformulation as a nonlinear singular integral equation in terms of the tangent angle.
It is this general approach to global bifurcation that we will pursue in the present work. However, when using conformal mappings in order to allow for overhanging waves and leaving the vorticity function general, such a reformulation becomes more involved. To get hands on this and still gain a derivative, in Lemma \ref{lma:Reformulation}, which lies somewhat at the heart of the present work, a Babenko-type reformulation of Bernoulli's equation is introduced.

Our paper is organised as follows: In Section \ref{sec:Preliminaries}, we state the problem, some preliminary tools, the functional analytic setting, and our reformulation. The justification of this reformulation is the content of Section \ref{sec:Reformulation}. Having the new operator equation at hand, we proceed in Section \ref{sec:LocalBifurcation} with the investigation of local bifurcation. To this end, we make use of the so-called good unknown and characterise the kernel and the range of the linearisation and, moreover, the transversality condition; see Lemmas \ref{lma:kernel} and \ref{lma:transversality_condition}. This characterisation, though not explicit, is complete under the assumption \eqref{ass:SL-spectrum} on the Dirichlet spectrum of a certain Sturm--Liouville operator. A corresponding result on local bifurcation is stated in Theorem \ref{thm:LocalBifurcation}. Afterwards, in order to cast a different light on the condition for local bifurcation obtained in Section \ref{sec:LocalBifurcation}, we first make some general conclusions in Section \ref{sec:general_conclusions}, and then show in Section \ref{sec:ConditionsLocalBifurcation_unidirectional} that the dispersion relation is equivalent to a certain Sturm--Liouville eigenvalue problem having a nontrivial solution, assuming that the semi-hodograph transformation can be applied at a trivial solution from which the curve shall bifurcate. Furthermore, we illustrate our condition with examples for the vorticity function, namely, constant and affine vorticity; see Section \ref{sec:ConditionsLocalBifurcation_examples}. Finally, we construct in Section \ref{sec:GlobalBifurcation} large amplitude solutions, applying the global bifurcation theorem of Rabinowitz, and investigate which alternatives for the global solution curve can occur; see Theorem \ref{thm:GlobalBifurcation}.

\section{Statement of the problem and preliminaries}\label{sec:Preliminaries}
Throughout this paper, for a given open set $\Omega\subset\R^n$, $n\in\N\coloneqq\{1,2,\ldots\}$, and parameters $k\in\N_0\coloneqq\N\cup\{0\}$, $0<\alpha\le 1$, we denote by $C^{k,\alpha}(\overline\Omega)$ the usual Hölder space (Lipschitz space if $\alpha=1$), that is, the space of functions on $\Omega$ having derivatives up to order $k$ such that all derivatives of order $k$ are Hölder continuous with parameter $\alpha$. These spaces are equipped with the usual norm
\[\|f\|_{C^{k,\alpha}(\overline\Omega)}=\sum_{|\beta|<k}\sup_\Omega |D^\beta f|+\sum_{|\beta|=k}\sup_{x,y\in\Omega,x\ne y}\frac{|D^\beta f(x)-D^\beta f(y)|}{|x-y|^\alpha}.\]
The index \enquote{loc} in, for example, $C_{\text{loc}}^{k,\alpha}(\overline\Omega)$ indicates that such functions are locally of class $C^{k,\alpha}$, that is, elements of $C^{k,\alpha}(\overline{\Omega'})$ for all open, bounded sets $\Omega'\subset\R^n$ such that $\overline{\Omega'}\subset\Omega$. Furthermore, for $\Omega$ and $k$ as above and $1\le p<\infty$, $W^{k,p}(\Omega)$ denotes the usual Sobolev space (Lebesgue space if $k=0$) equipped with the standard norm
\[\|f\|_{W^{k,p}(\Omega)}=\left(\sum_{|\beta|\le k}\int_\Omega|D^\beta f|^p\right)^{1/p}.\]
We use the usual abbreviations $L^p=W^{0,p}$ and $H^k=W^{k,2}$. Henceforth, $0<\alpha<1$ is arbitrary and fixed. Moreover, we write $\SL$ for the operator evaluating functions $f$, which are defined (at least) for $(x,y)\in\R\times\{0\}$, at $y=0$, that is, $\SL f$ is the function $x\mapsto f(x,0)$. Furthermore, we denote (partial) derivatives by lower indices; for example, $f_x$ is the (partial) derivative of $f$ with respect to $x$.

The equations (cf. \cite{Wahlen06b}, for example) we want to solve are
\begin{subequations}\label{eq:OriginalEquations}
	\begin{align}
	\Delta\psi&=-\gamma(\psi)&\text{in }\Omega,\label{eq:OriginalEquations_Poisson}\\
	\frac{|\nabla\psi|^2}{2}-\sigma\kappa+g(y-h)&=Q&\text{on }S\label{eq:OriginalEquations_Bernoulli},\\
	\psi&=0&\text{on }S,\label{eq:OriginalEquations_KinematicTop}\\
	\psi&=-m&\text{on }y=0\label{eq:OriginalEquations_KinematicBottom}.
	\end{align}
\end{subequations}
Here, $\psi$ is the stream function, which is assumed to be $L$-periodic in $x$ (where the period $L>0$ is fixed; we abbreviate $\nu\coloneqq 2\pi/L$) and which provides the velocity field $(\psi_y,-\psi_x)$ in a frame moving at a constant wave speed. Moreover, $\gamma$ is the vorticity function satisfying
\[\gamma\in C_{\text{loc}}^{2,1}(\R),\quad\|\gamma'\|_\infty<\infty,\]
$g>0$ is the gravitational constant, $\sigma>0$ is the constant coefficient of surface tension, $\Omega$ is the a priori unknown $L$-periodic fluid domain bounded by the free surface $S$ and flat bed $\R\times\{0\}$, $\kappa$ is the mean curvature of $S$, $h>0$ is the conformal mean depth of $\Omega$, and $Q,m\in\R$ are constants; in physical language, $Q$ is related to the hydraulic head $E$ and the (constant) atmospheric pressure $p_{\mathrm{atm}}$ via $Q=E-p_{\mathrm{atm}}-gh$, and $m$ is the relative mass flux. In order to clarify the notion of conformal mean depth and to provide a very important preliminary result on conformal mappings, we pause here to state the following lemma, where $\Ch$, the $L$-periodic Hilbert transform on a strip with depth $h$ given by
\[
\Ch u(x)=-i\sum_{k\ne 0} \coth(k\nu h) \hat u_k e^{ik\nu x}
\]
for a $L$-periodic function
\[u(x)=\sum_{k\ne 0} \hat u_k e^{ik\nu x}\]
with zero average, appears first.
\begin{lemma}\label{lma:ConformalMapping}
	Let $\Omega\subset\R^2$ be a $L$-periodic strip-like domain, that is, a domain contained in the upper half plane such that its boundary consists of the real axis $\R\times\{0\}$ and a not self-intersecting curve $S=\{(u(s),v(s)):s\in\R\}$ with $u(s+L)=u(s)+L$, $v(s+L)=v(s)$, $s\in\R$. Then:
	\begin{enumerate}[label=(\roman*)]
		\item There exists a unique positive number $h$ such that there exists a conformal mapping $H=U+iV$ from the strip $\Omega_h\coloneqq\R\times(-h,0)$ to $\Omega$ which admits an extension as a homeomorphism between the closures of these domains, with $\R\times\{0\}$ being mapped onto S and $\R\times\{-h\}$
		being mapped onto $\R\times\{0\}$, and such that $U(x+L,y)=U(x,y)+L$, $V(x+L,y)=V(x,y)$, $(x,y)\in\overline{\Omega_h}$.
		\item The conformal mapping $H$ is unique up to translations in the variable $x$ (in the preimage and the image).
		\item $U$ and $V$ are (up to translations in the variable $x$) uniquely determined by $w=V(\cdot,0)-h$ as follows: $V$ is the unique ($L$-periodic) solution of
		\begin{subequations}\label{eq:BVP_for_V}
		\begin{align}
			\Delta V&=0&\mathrm{in\ }\Omega_h,\\
			V&=w+h&\mathrm{on\ }y=0,\\
			V&=0&\mathrm{on\ }y=-h,
		\end{align}
		\end{subequations}
		and $U$ is the (up to a real constant unique) harmonic conjugate of $V$. Furthermore, after a suitable horizontal translation, $S$ can be parametrised by
		\[S=\{(x+(\Ch w)(x),w(x)+h):x\in\R\}\]
		and it holds that
		\begin{align}\label{eq:nablaV_on_top}
			\SL\nabla V=(1+\Ch w',w').
		\end{align}
		\item If $S$ is of class $C^{1,\beta}$ for some $\beta>0$, then $U,V\in C^{1,\beta}(\overline{\Omega_h})$ and
		\[|dH/dz|^2=|\nabla V|^2\neq 0\mathrm{\ in\ }\overline{\Omega_h}.\]
	\end{enumerate}
\end{lemma}
\begin{proof}
	See \cite[Thm. 2.2., Appx. A]{ConsVarva11}. (That horizontal translations are the only degree of freedom is reminiscent of the fact that the automorphisms of annuli leaving their boundary circles invariant are exactly the rotations.)
\end{proof}

In \eqref{eq:OriginalEquations}, we demand that $\Omega$ is a $L$-periodic strip-like domain of class $C^{1,\alpha}$ with conformal mean depth $h$ so that it is, due to Lemma \ref{lma:ConformalMapping}, the conformal image of $\Omega_h$ and the free surface $S$ is determined by some $L$-periodic $w$ of class $C^{1,\alpha}$ with zero average over one period satisfying
\[(1+\Ch w')^2+w'^2\ne 0\text{ on }\R\]
via
\begin{align}\label{eq:Sw}
	S=\{(x+(\Ch w)(x),w(x)+h):x\in\R\}.
\end{align}
Throughout this paper, $h>0$ is fixed.

In order to make sense of the mean curvature
\[\kappa=\kappa[w](x)=\frac{(1+(\Ch w')(x))w''(x)-w'(x)(\Ch w'')(x)}{((1+(\Ch w')(x))^2+w'(x)^2)^{3/2}}\]
at $(x+(\Ch w)(x),w(x)+h)$, the function $w$ should be of class $C^{2,\alpha}$. Since $\Omega$ should be a strip-like domain, we have to exclude self-intersection of the wave and intersection of the surface with the bed, which reads in terms of $w$ as
\begin{subequations}\label{eq:additional_requirements}
\begin{gather}
	x\mapsto(x+(\Ch w)(x),w(x)+h)\text{ is injective on }\R,\\
	w>-h\text{ on }\R.
\end{gather}	
\end{subequations}

In order to identify a curve of trivial solutions, we have to solve the boundary value problem
\[\psi_{yy}=-\gamma(\psi)\text{ on }[-h,0],\quad\psi(0)=0,\quad\psi(-h)=-m\]
for $\psi=\psi(y)$; then $\psi(\cdot-h)$ solves \eqref{eq:OriginalEquations} with $w=0$ (that is, $\Omega=\R\times(0,h)$). However, in general, a solution to this problem neither has to exist nor be unique. Thus, we introduce a new parameter $\lambda\in\R$, which will later serve as the bifurcation parameter, and prescribe $\psi_y(0)=\lambda$. Notice that $\lambda$ can be interpreted as the (horizontal) velocity at the surface (or, rather, the relative velocity as everything is usually written down in a moving frame). The trivial solution corresponding to $\lambda$ is then $\psi=\psi^\lambda$, which is defined to be the unique solution of the Cauchy problem
\begin{align}\label{eq:trivial_sol}
	\psi_{yy}^\lambda=-\gamma(\psi^\lambda)\text{ on }[-h,0],\quad\psi^\lambda(0)=0,\quad\psi_y^\lambda(0)=\lambda.
\end{align}
Indeed, there exists a unique solution due to the global Lipschitz continuity of $\gamma$. We therefore view $m$ not as a parameter, but as a function $m=m(\lambda)$ of $\lambda$ defined by
\begin{align}\label{eq:def_m}
	m(\lambda)\coloneqq -\psi^\lambda(-h).
\end{align}
At this point we mention that the assumption that $\gamma$ be globally Lipschitz continuous is only needed to ensure that all trivial solutions exist on $[-h,0]$; we could therefore relax this assumption and demand that all trivial solutions under consideration exist on $[-h,0]$, which is, however, not very explicit.

We now introduce the functional-analytic setting: The unknowns are $(w,\phi)\in\U\subset X$, where we denote 
\begin{align*}
	X&\coloneqq C_{0,\per,\e}^{1,\alpha}(\R)\times\left\{\phi\in C_{\per,\e}^{0,\alpha}(\overline{\Omega_h})\cap H_\per^1(\Omega_h):\phi=0\text{ on }y=0\text{ and on }y=-h\right\},\\
	\U&\coloneqq\left\{(w,\phi)\in X:(1+\Ch w')^2+w'^2>0\text{ on }\R\right\},\\
	\Om&\coloneqq\R\times\U.
\end{align*}
Here, the indices \enquote{$\per$}, \enquote{$\e$}, and \enquote{$0$} denote $L$-periodicity, evenness (in $x$ with respect to $x=0$), and zero average over one period. Moreover, $X$ is equipped with the norm
\begin{align*}
	\|(w,\phi)\|_X&\coloneqq\|w\|_{C_\per^{1,\alpha}(\R)}+\|\phi\|_{C_\per^{0,\alpha}(\overline{\Omega_h})\cap H_\per^1(\Omega_h)};\\
	\|w\|_{C_\per^{1,\alpha}(\R)}&\coloneqq\|w\|_{C^{1,\alpha}([0,L])},\\
	\|\phi\|_{C_\per^{0,\alpha}(\overline{\Omega_h})\cap H_\per^1(\Omega_h)}&\coloneqq\|\phi\|_{C_\per^{0,\alpha}(\overline{\Omega_h})}+\|\phi\|_{H_\per^1(\Omega_h)}\coloneqq\|\phi\|_{C^{0,\alpha}(\overline{\Omega_h^*})}+\|\phi\|_{H^1(\Omega_h^*)}.
\end{align*}
Here and in the following, $\Omega^*$ denotes one copy of a $L$-periodic domain $\Omega\subset\R^2$. In general, we will write
\[\|\cdot\|_{C_\per^{k,\alpha}(\R)}\coloneqq\|\cdot\|_{C^{k,\alpha}([0,L])},\quad\|\cdot\|_{C_\per^k(\R)}\coloneqq\sum_{j=0}^k\|\cdot^{(j)}\|_\infty,\quad\|\cdot\|_{C_\per^{k,\alpha}(\overline\Omega)}\coloneqq\|\cdot\|_{C^{k,\alpha}(\overline{\Omega^*})}\]
for $k\in\N_0$ and $\Omega$ as above, where $C^k$ is the space of $k$-times continuously differentiable functions as usual.

Given $w\in C_{0,\per}^{1,\alpha}(\R)$, we denote by $V$ the unique solution of \eqref{eq:BVP_for_V} in $C_\per^{1,\alpha}(\overline{\Omega_h})$. Provided \eqref{eq:additional_requirements}, $V$ gives rise to a conformal mapping $H=U+iV$ from $\Omega_h$ to $\Omega=\Omega_w$, the surface of which being determined by \eqref{eq:Sw}, in view of Lemma \ref{lma:ConformalMapping}. Notice that $V$ is explicitly given by
\begin{align}\label{eq:V_explicit}
	V(x,y)=y+h+\sum_{k\neq 0}\hat w_k e^{ikx} \frac{\sinh(k\nu(y+h))}{\sinh(k\nu h)},
\end{align}
where the $\hat w_k$ are the Fourier coefficients of $w$. We sometimes denote $V=V[w+h]$ (and likewise also for $H$ and $U$) if we want to express the dependency of $V$ on the boundary condition at $y=0$ explicitly. It is clear that $V$ is analytic in $w$. In general, for boundary conditions $v$ at $y=0$ not necessarily having the form $w+h$, \eqref{eq:V_explicit} reads
\begin{align*}
	V[v](x,y)=\frac{\langle v\rangle}{h}(y+h)+\sum_{k\neq 0}\hat v_k e^{ikx}\frac{\sinh(k\nu (y+h))}{\sinh(k\nu h)}.
\end{align*}
Here and in the following, $\langle f\rangle$ denotes the average of a $L$-periodic function $f$ over one period.
	
For $(\lambda,w,\phi)\in\Om$, we write $\A(\lambda,w,\phi)=\tilde\psi$ for the unique solution of
\begin{align*}
	\Delta\tilde\psi&=-\gamma(\phi+\psi^\lambda)|\nabla V|^2+\gamma(\psi^\lambda)&\text{in }\Omega_h,\\
	\tilde\psi&=0&\text{on }y=0,\\
	\tilde\psi&=0&\text{on }y=-h.
\end{align*}
in $C_\per^{2,\alpha}(\overline{\Omega_h})$. Notice that this is equivalent to
\begin{align*}
	\Delta\psi&=-\gamma((\phi+\psi^\lambda)\circ H^{-1})&\text{in }\Omega_w,\\
	\psi&=0&\text{on }S_w,\\
	\psi&=-m(\lambda)&\text{on }y=0,
\end{align*}
where $\psi=(\tilde\psi+\psi^\lambda)\circ H^{-1}$. Thus, $\A(\lambda,w,\phi)=\phi$ is equivalent to the statement that $\psi=(\phi+\psi^\lambda)\circ H^{-1}$ solves \eqref{eq:OriginalEquations_Poisson}, \eqref{eq:OriginalEquations_KinematicTop}, and \eqref{eq:OriginalEquations_KinematicBottom} with $\Omega=H(\Omega_h)$ and $m=m(\lambda)$, provided \eqref{eq:additional_requirements}. In particular, the points $(\lambda,0,0)$ are identified as the trivial solutions. Moreover, clearly $V$ and thus $\A(\lambda,w,\phi)$ are even in $x$.

Let us denote
\begin{align*}
	\K(w)&\coloneqq((1+\Ch w')^2+w'^2)^{1/2},\\
	Q(\lambda,w,\phi)&\coloneqq\frac{\left\langle\K(w)\left(\frac{1}{2\K(w)^2}\SL|\nabla(\A(\lambda,w,\phi)+\psi^\lambda)|^2+gw\right)\right\rangle}{\langle\K(w)\rangle},\\
	R(\lambda,w,\phi)&\coloneqq\sigma^{-1}\K(w)\left(\frac{1}{2\K(w)^2}\SL|\nabla(\A(\lambda,w,\phi)+\psi^\lambda)|^2+gw-Q(\lambda,w,\phi)\right)
\end{align*}
for $(\lambda,w,\phi)\in\Om$. Due to the appearance of the squares, it is evident that all three expressions are also even in $x$.

We denote by $\partial_x^{-2}\colon C_{0,\per}^{0,\alpha}(\R)\to C_{0,\per}^{2,\alpha}(\R)$ the inverse operation to twice differentiation; explicitly, it holds that
\[\partial_x^{-2}f=\Pro\left(x\mapsto\int_0^x\int_0^sf(t)\,dtds-\frac{x}{L}\int_0^L\int_0^sf(t)\,dtds\right),\]
where $\Pro$ is the projection onto the space of functions with zero average. From a Fourier point of view, $\partial_x^{-2}$ is the multiplier with symbol $-\frac{1}{(k\nu)^2}$.
	
We reformulate the original problem \eqref{eq:OriginalEquations} as 
\begin{align}\label{eq:F=0}
	F(\lambda,w,\phi)=0
\end{align}
for $(\lambda,w,\phi)\in\Om$, where
\begin{gather*}
	F\colon\Om\to X,\;F(\lambda,w,\phi)=(w,\phi)-\M(\lambda,w,\phi),\\
	\M(\lambda,w,\phi)\coloneqq\left(\partial_x^{-2}((1+\Ch w')R(\lambda,w,\phi)+w'\Ch R(\lambda,w,\phi)),\A(\lambda,w,\phi)\right).
\end{gather*}
We point out that (in particular the first component of) $\M$ is well-defined since $R(\lambda,w,\phi)$ has zero average by definition and $\langle(\Ch f_1)f_2+f_1\Ch f_2\rangle=0$ for any $f_1,f_2\in C_{0,\per}^{0,\alpha}(\R)$ as $-i\coth(k\nu h)$, the symbol of $\Ch$, is odd, and, moreover, $\Ch w'$, $R(\lambda,w,\phi)$ are even and $w'$, $\Ch R(\lambda,w,\phi)$ are odd so that $\M_1(\lambda,w,\phi)$ is even in $x$.

\section{On the reformulation}\label{sec:Reformulation}
We now prove two important lemmas which justify and motivate this new reformulation. A similar version of the following lemma in the case of constant vorticity was obtained in \cite{deBoeck14a}.
\begin{lemma}\label{lma:Reformulation}
	Under the assumption \eqref{eq:additional_requirements}, a tuple $(\lambda,w,\phi)\in\Om$ solves \eqref{eq:F=0} if and only if $w$ is of class $C^{2,\alpha}$, $\Omega_w$, the surface of which being determined by \eqref{eq:Sw}, is of class $C^{1,\alpha}$, and $(w,(\phi+\psi^\lambda)\circ H^{-1})$, where the conformal mapping $H=U+iV\colon\Omega_h\to\Omega_w$ is uniquely determined by $w$, solves \eqref{eq:OriginalEquations} with $\Omega=\Omega_w$, $Q=Q(\lambda,w,\phi)$, and $m=m(\lambda)$.
\end{lemma}
\begin{proof}
	As was already observed, $\phi=\A(\lambda,w,\phi)$ is equivalent to that \eqref{eq:OriginalEquations_Poisson}, \eqref{eq:OriginalEquations_KinematicBottom}, \eqref{eq:OriginalEquations_KinematicTop} are solved by $\psi=(\phi+\psi^\lambda)\circ H^{-1}$ with $m=m(\lambda)$. It is also clear that $w$ is of class $C^{2,\alpha}$ and $\Omega_w$ is of class $C^{1,\alpha}$ provided $(\lambda,w,\phi)\in\Om$ solves \eqref{eq:F=0}. Moreover, we rewrite the Bernoulli equation \eqref{eq:OriginalEquations_Bernoulli} for $\psi=(\phi+\psi^\lambda)\circ H^{-1}$:
	\begin{align}
		&\frac{(1+\Ch w')w''-w'\Ch w''}{(1+\Ch w')^2+w'^2}\nonumber\\
		&=\sigma^{-1}((1+\Ch w')^2+w'^2)^{1/2}\left(\SL\left[\frac{|\nabla(\A(\lambda,w,\phi)+\psi^\lambda)|^2}{2|\nabla V|^2}\right]+gw-Q\right)\eqqcolon R.\label{eq:OriginalBernoulliFlat}
	\end{align}
	First notice that
	\begin{align}\label{eq:length_of_nablaV_on_top}
		\SL|\nabla V|^2=\K(w)^2
	\end{align}
	due to \eqref{eq:nablaV_on_top}.
	
	Now let $u$ be any $L$-periodic real-valued function with zero average, written as
	\[
	u(x)=\sum_{k\ne 0} \hat u_k e^{ik\nu x},
	\]
	with $\hat u_{-k}=\overline{\hat u_k}$.
	Then, $\mathcal R_h$ is defined by
	\begin{align*}
	\mathcal R_h u(x+iy)
	&=-\sum_{k\ne 0} \hat u_k e^{ik\nu(x+i(y+h))} \frac{1}{\sinh(k\nu h)}
	\\
	&=\sum_{k\ne 0} \hat u_k e^{ik\nu x} \frac{\sinh(k\nu(y+h))-\cosh(k\nu(y+h))}{\sinh(k\nu h)} \\
	&=\sum_{k\ne 0} \hat u_k e^{ik\nu x}\frac{\sinh(k\nu(y+h))}{\sinh(k\nu h)}
	+i \sum_{k\ne 0} \hat u_k e^{ik\nu x}\frac{i\cosh(k\nu(y+h))}{\sinh(k\nu h)}\\
	&=\vartheta+i\varphi,
	\end{align*}
	where $\vartheta$ is harmonic with $\SL\vartheta=u$, $\vartheta|_{y=-h}=0$ and
	$\varphi_y|_{y=-h}=0$. Then formally $\mathcal R_h u$ is analytic with
	\[
	\mathcal R_h u(x)=u(x)-i\Ch u(x).
	\]
	
	In order to rewrite the Bernoulli equation in a suitable form,
	introduce two analytic functions defined by
	\[
	A+iB=1+i\mathcal R_h w'
	\]
	and
	\[
	C+iD=i\mathcal R_h w''
	\]
	so that
	\[
	\SL[A+iB]=(1+\Ch w')+iw'
	\] 
	and 
	\[
	\SL[C+iD]=\Ch w''+i w''.
	\]
	From \eqref{eq:V_explicit} and Lemma \ref{lma:ConformalMapping}(iv)---notice that in both statements to be proved equivalent the surface is of class $C^{1,\alpha}$---we infer that $A+iB=V_y+iV_x=\frac{d}{dz}H\neq 0$ everywhere.
	Then,
	\[
	\SL\left[\text{Im}\left\{\frac{C+iD}{A+iB}\right\}\right]=\frac{(1+\Ch w')w''-w'\Ch w''}{(1+\Ch w')^2+w'^2}=R.
	\]
	Note that 
	\[
	\text{Im}\left\{\frac{C+iD}{A+iB}\right\}=\frac{AD-BC}{A^2+B^2}
	\]
	and
	\[
	\text{Re}\left\{\frac{C+iD}{A+iB}\right\}=\frac{AC+BD}{A^2+B^2}
	\]
	so that $\text{Im}\left\{\frac{C+iD}{A+iB}\right\}$ vanishes at $y=-h$ (because $B$ and $D$ do), while $\text{Re}\left\{\frac{C+iD}{A+iB}\right\}$ is $L$-periodic in $x$.
	It follows from the Cauchy-Riemann equations that 
	\[
	\SL\left[\text{Im}\left\{\frac{C+iD}{A+iB}\right\}\right]
	\]
	has zero average. 
	Indeed if $\vartheta+i\varphi$ is $L$-periodic and analytic with $\varphi|_{y=-h}=0$ then
	\[
	\frac{d}{dy} \int_{0}^{L} \varphi(x,y) \, dx=\int_0^{L} \varphi_y(x,y)\, dx=\int_0^{L} \vartheta_x(x,y)\, dx
	=\vartheta(L,y)-\vartheta(0,y)=0
	\]
	and
	\[
	\int_{0}^{L} \varphi(x,-h) \, dx=0.
	\]
	Therefore, $R$ has zero average; in particular, $Q=Q(\lambda,w,\phi)$ and $R=R(\lambda,w,\phi)$. Moreover, we can define an analytic function $E+iG\coloneqq i\mathcal R_h R$, and
	\[
	\frac{C+iD}{A+iB}=E+iG +K
	\]
	for some real constant $K$ (the imaginary parts are harmonic, periodic and equal on the top and bottom, and therefore equal everywhere). 
	The constant $K$ is the average over one period of
	\begin{align*}
		\SL\left[\text{Re}\left\{\frac{C+iD}{A+iB}\right\}\right]&=\SL\left[\frac{AC+BD}{A^2+B^2}\right]
		=\frac{(1+\Ch w')\Ch w''+w'w''}{(1+\Ch w')^2+w'^2}\\
		&=
		\frac12 \frac{d}{dx} \ln ((1+\Ch w')^2+w'^2).
	\end{align*}
	Hence, $K=0$ and
	\[
	C+iD=(E+iG)(A+iB)
	\]
	and in particular
	\begin{equation}
	\label{eq:nice}
	w''=\SL D=\SL[AG+BE]=(1+\Ch w')R+w'\Ch R,
	\end{equation}
	which is exactly the equation $F_1(\lambda,w,\phi)=0$, that is, the first component of \eqref{eq:F=0}.
	
	In order to go in the opposite direction, suppose that we start with \eqref{eq:nice} for $R=R(\lambda,w,\phi)$, which has zero average. Then
	we get
	\[
	C+iD=(E+iG)(A+iB)+\tilde K,
	\]
	for some real constant $\tilde K$.
	Hence,
	\[
	\frac{C+iD}{A+iB}=E+iG+ \frac{\tilde K}{A+iB}.
	\]
	We use this to prove that $\tilde K=0$. Indeed, we already know that 
	$\SL\left[\text{Re}\left\{\frac{C+iD}{A+iB}\right\}\right]$ has zero average. Hence, the same is true over any horizontal line.
	Furthermore, $E$ has zero average over any horizontal line. Therefore 
	$\text{Re}\left\{\frac{\tilde K}{A+iB}\right\}$ has zero average over any horizontal line. Evaluating at $y=-h$ we find that $B$ vanishes, and thus the average of $\frac{\tilde K}{A|_{y=-h}}$ vanishes. But clearly this is impossible unless $\tilde K=0$.
	Finally, we get
	\[
	\frac{(1+\Ch w')w''-w'\Ch w''}{(1+\Ch w')^2+w'^2}=\SL\left[\text{Im}\left\{\frac{C+iD}{A+iB}\right\}\right]=\SL G=R,
	\]
	which is exactly \eqref{eq:OriginalEquations_Bernoulli} with constant $Q(\lambda,w,\phi)$.
\end{proof}
\begin{lemma}\label{lma:M_prop}
	The operator $\M$ (in particular, $\A$) and thus $F$ is of class $C^2$ on $\Om$. Moreover, for every $\varepsilon>0$ the operator $\M$ is compact on the set
	\[\Om_\varepsilon\coloneqq\R\times\U_\varepsilon\coloneqq\R\times\left\{(w,\phi)\in\U:(1+\Ch w')^2+w'^2\ge\varepsilon\emph{ on }\R\right\}.\]
\end{lemma}
\begin{proof}
	The other operations in the definition of $\M$ being smooth, the property that $\M$ is of class $C^2$ follows from the property that $\A$ is of class $C^2$; this, in turn, is guaranteed by the assumption $\gamma\in C_{\text{loc}}^{2,1}(\R)$. Now let $(\lambda,w,\phi)\in\Om_\varepsilon$ be arbitrary. In the following, we apply standard Schauder estimates; the quantity $c$ can change from line to line, but is always shorthand for a certain expression in its arguments which remains bounded for bounded arguments. First we have
	\begin{align}\label{eq:est_V}
		\|V\|_{C_\per^{1,\alpha}(\overline{\Omega_h})}\le c\left(\|w\|_{C_\per^{1,\alpha}(\R)}\right).
	\end{align}
	Thus and since $\psi^\lambda$ is of class $C^1$ with respect to $\lambda$, we see that
	\begin{align*}
		\|\A(\lambda,w,\phi)\|_{C_\per^{2,\alpha}(\overline{\Omega_h})}&\le c\left(\|\phi\|_{C_\per^{0,\alpha}(\overline{\Omega_h})},\|V\|_{C_\per^{1,\alpha}(\overline{\Omega_h})},|\lambda|\right)\\
		&\le c\left(\|\phi\|_{C_\per^{0,\alpha}(\overline{\Omega_h})},\|w\|_{C_\per^{1,\alpha}(\R)},|\lambda|\right).
	\end{align*}
	This shows that $\M_2$, the second component of $\M$, is compact on $\Om_\varepsilon$ due to the compact embedding of $C_\per^{2,\alpha}(\overline{\Omega_h})$ in $C_\per^{0,\alpha}(\overline{\Omega_h})$ and in $H_\per^1(\Omega_h)$. As for $\M_1$, we proceed with $R$ and find that
	\begin{align}\label{eq:est_R}
		\|R(\lambda,w,\phi)\|_{C_\per^{0,\alpha}(\R)}&\le c\left(\varepsilon^{-1},\|w\|_{C_\per^{1,\alpha}(\R)},\|\A(\lambda,w,\phi)\|_{C_\per^{1,\alpha}(\overline{\Omega_h})},|\lambda|\right)\nonumber\\
		&\le c\left(\varepsilon^{-1},\|\phi\|_{C_\per^{0,\alpha}(\overline{\Omega_h})},\|w\|_{C_\per^{1,\alpha}(\R)},|\lambda|\right)
	\end{align}
	since $\K(w)\ge\varepsilon^{1/2}$ and $\Ch$ is a bounded operator from $C^{0,\alpha}_{0,\per}(\R)$ (equipped with $\|\cdot\|_{C_\per^{0,\alpha}(\R)}$) to itself. Therefore,
	\begin{align}\label{eq:est_M1}
		\|\M_1(\lambda,w,\phi)\|_{C_\per^{2,\alpha}(\R)}\le c\left(\varepsilon^{-1},\|\phi\|_{C_\per^{0,\alpha}(\overline{\Omega_h})},\|w\|_{C_\per^{1,\alpha}(\R)},|\lambda|\right).
	\end{align}
	Hence, $\M_1$ is compact on $\Om_\varepsilon$ since $C_\per^{2,\alpha}(\R)$ is compactly embedded in $C_\per^{1,\alpha}(\R)$.
\end{proof}

\section{Local bifurcation}\label{sec:LocalBifurcation}
The goal of this section is to apply the following local bifurcation theorem of Crandall--Rabinowitz \cite[Thm. I.5.1]{Kielhoefer}.
\begin{theorem}\label{thm:CrandallRabinowitz}
	Let $X$ be a Banach space, $U\subset\R\times X$ open, and $F\colon U\to X$ have the property $F(\cdot,0)=0$. Assume that there exists $\lambda_0\in\R$ such that $F$ is of class $C^2$ in an open neighbourhood of $(\lambda_0,0)$, and suppose that $F_x(\lambda_0,0)$ is a Fredholm operator with index zero and one-dimensional kernel spanned by $x_0\in X$, and that the transversality condition $F_{\lambda x}(\lambda_0,0)x_0\notin\im F_x(\lambda_0,0)$ holds. Then there exists $\varepsilon>0$ and a $C^1$-curve $(-\varepsilon,\varepsilon)\ni s\mapsto(\lambda(s),x(s))$ with $(\lambda(0),x(0))=(\lambda_0,0)$ and $x(s)\neq 0$ for $s\neq 0$, and $F(\lambda(s),x(s))=0$. Moreover, all solutions of $F(\lambda,x)=0$ in a neighbourhood of $(\lambda_0,0)$ are on this curve or are trivial. Furthermore, the curve admits the asymptotic expansion $x(s)=sx_0+o(s)$.
\end{theorem}
In order to apply this theorem, we have to study the linearised operator and also the transversality condition. This will be the content of the next subsections.
\subsection{Computing derivatives}
We first want to calculate the partial derivative $F_{(w,\phi)}$ and, in particular, its value at a trivial solution. To this end, we introduce the abbreviation
\[B=B(\lambda,w,\phi)=\frac{1}{2\K(w)^2}\SL|\nabla(\A(\lambda,w,\phi)+\psi^\lambda)|^2.\] With this we have
\[Q=\frac{\langle\K(B+gw)\rangle}{\langle\K\rangle},\quad\sigma R=\K(B+gw-Q).\]
We first evaluate these quantities at a trivial solution $(\lambda,0,0)$. We have 
\[V=y+h,\,\nabla V=\begin{pmatrix}0\\1\end{pmatrix},\,\K=1,\,\A=0,\,\SL\nabla(\A+\psi^\lambda)=\begin{pmatrix}0\\\lambda\end{pmatrix},\,B=\frac{\lambda^2}{2},\,Q=\frac{\lambda^2}{2},\, R=0.\]
Let now $(\lambda,w,\phi)$ be arbitrary and $(\delta w,\delta\phi)$ be a direction. First consider $V_w(w)\delta w$, the partial derivative of $V$ with respect to $w$ evaluated at $w$ and applied to the direction $\delta w$, and abbreviate $V_w=V_w(w)\delta w$; we will also use this abbreviation similarly for other expressions and derivatives when there is no possibility of confusion. Now $V_w$ is the unique solution of
\begin{align*}
	\Delta V_w&=0&\text{in }\Omega_h,\\
	V_w&=\delta w&\text{on }y=0,\\
	V_w&=0&\text{on }y=-h.
\end{align*}
Next, $\A_w$ and $\A_\phi$ are the unique solutions of
\begin{align*}
\left.\begin{aligned}
	\Delta\A_w&=-2\gamma(\phi+\psi^\lambda)\nabla V\cdot\nabla V_w&\text{in }\Omega_h,\\
	\A_w&=0&\text{on }y=0,\\
	\A_w&=0&\text{on }y=-h,
\end{aligned}
\quad\middle|\quad
\begin{aligned}
\Delta\A_w&=-2\gamma(\psi^\lambda)\partial_yV_w&\text{in }\Omega_h,\\
\A_w&=0&\text{on }y=0,\\
\A_w&=0&\text{on }y=-h,
\end{aligned}\right.
\end{align*}
and
\begin{align*}
\left.\begin{aligned}
	\Delta\A_\phi&=-\gamma'(\phi+\psi^\lambda)\delta\phi|\nabla V|^2&\text{in }\Omega_h,\\
	\A_\phi&=0&\text{on }y=0,\\
	\A_\phi&=0&\text{on }y=-h,
\end{aligned}
\quad\middle|\quad
\begin{aligned}
	\Delta\A_\phi&=-\gamma'(\psi^\lambda)\delta\phi&\text{in }\Omega_h,\\
	\A_\phi&=0&\text{on }y=0,\\
	\A_\phi&=0&\text{on }y=-h.
\end{aligned}\right.
\end{align*}
Here and in the following, the left columns are the equations if $(\lambda,w,\phi)$ is arbitrary, and the right columns show the simplifications if $(\lambda,w,\phi)=(\lambda,0,0)$ is a trivial solution. Moreover, it holds that
\begin{align*}
\left.\begin{aligned}
	\K_w&=\K^{-1}\left((1+\Ch w')\Ch\delta w'+w'\delta w'\right),\\
	B_w&=\SL\Big[-\K^{-3}\K_w|\nabla(\A+\psi^\lambda)|^2\\
	&\phantom{=\;}+\K^{-2}\nabla(\A+\psi^\lambda)\cdot\nabla\A_w\Big],\\
	B_\phi&=\SL[\K^{-2}\nabla(\A+\psi^\lambda)\cdot\nabla\A_\phi],\\
	Q_w&=\frac{\langle\K_w(B+gw)+\K(B_w+g\delta w)\rangle}{\langle\K\rangle}\\
	&\phantom{=\;}-\frac{\langle\K(B+gw)\rangle\langle\K_w\rangle}{\langle\K\rangle^2},\\
	Q_\phi&=\frac{\langle\K B_\phi\rangle}{\langle\K\rangle},\\
	\sigma R_w&=\K_w(B+gw-Q)+\K(B_w+g\delta w-Q_w),\\
	\sigma R_\phi&=\K(B_\phi-Q_\phi)
\end{aligned}
\quad\middle|\quad
\begin{aligned}
	\K_w&=\Ch\delta w',\\
	B_w&=-\lambda^2\Ch\delta w'+\lambda\SL\partial_y\A_w,\\
	B_\phi&=\lambda\SL\partial_y\A_\phi,\\
	Q_w&=\langle B_w\rangle,\\
	Q_\phi&=\langle B_\phi\rangle,\\
	\sigma R_w&=B_w+g\delta w-Q_w,\\
	\sigma R_\phi&=B_\phi-Q_\phi
\end{aligned}\right.
\end{align*}
Notice that we exploited the fact that $\Ch$ maps functions with zero average to functions with zero average. Finally, we conclude
\begin{align*}
	F_w&=\left(\delta w-\partial_x^{-2}(R\Ch\delta w'+(1+\Ch w')R_w+\delta w'\Ch R+w'\Ch R_w),-\A_w\right),\\
	F_\phi&=\left(-\partial_x^{-2}((1+\Ch w')R_\phi+w'\Ch R_\phi),\delta\phi-\A_\phi\right)
\end{align*}
in general, and
\begin{subequations}\label{eq:Fder_trivial}
\begin{align}
	F_w&=\left(\delta w-\sigma^{-1}\partial_x^{-2}(\lambda\Pro\SL\partial_y\A_w-\lambda^2\Ch\delta w'+g\delta w),-\A_w\right),\\
	F_\phi&=\left(-\sigma^{-1}\lambda\partial_x^{-2}(\Pro\SL\partial_y\A_\phi),\delta\phi-\A_\phi\right)
\end{align}
\end{subequations}
evaluated at a trivial solution.

We note that, since $\M$ is compact due to Lemma \ref{lma:M_prop}, also its derivative is compact. Thus, $F_{(w,\phi)}=\mathrm{Id}-\M_{(w,\phi)}$ is a compact perturbation of the identity and hence a Fredholm operator with index zero.

\subsection{The good unknown}
Before we proceed with the investigation of local bifurcation, we first introduce an isomorphism, which facilitates the computations later and is sometimes called $\T$-isomorphism in the literature (for example, in \cite{EhrnEschWahl11,Varholm20}). The discovery of the importance of such a new variable (here $\theta$) goes back to Alinhac \cite{Alinhac89}, who called it the \enquote{good unknown} in a very general context, and Lannes \cite{Lannes05}, who introduced it in the context of water wave equations.
\begin{lemma}
	Let
	\[Y\coloneqq\left\{\theta\in C_{\per,\e}^{0,\alpha}(\overline{\Omega_h})\cap H_\per^1(\Omega_h):\SL\theta\in C_{0,\per,\e}^{1,\alpha}(\R),\theta=0\mathrm{\ on\ }y=-h\right\},\]
	equipped with
	\[\|\theta\|_Y\coloneqq\|\theta\|_{C_\per^{0,\alpha}(\overline{\Omega_h})\cap H_\per^1(\Omega_h)}+\|\SL\theta\|_{C_\per^{1,\alpha}(\R)},\]
	and assume that $\lambda\neq 0$. Then
	\[\T(\lambda)\colon Y\to X,\quad\T(\lambda)\theta=\left(-\frac{\SL\theta}{\lambda},\theta-\frac{\psi_y^\lambda}{\lambda}V[\SL\theta]\right)\]
	is an isomorphism. Its inverse is given by
	\[[\T(\lambda)]^{-1}(\delta w,\delta\phi)=\delta\phi-\psi_y^\lambda V[\delta w].\]
\end{lemma}
\begin{proof}
	Both $\T(\lambda)$ and $[\T(\lambda)]^{-1}$ are well-defined, and a simple computation shows that they are inverse to each other.
\end{proof}

Let us now consider a trivial solution $(\lambda,0,0)$. In view of the isomorphism $\T$, we introduce
\[\LL(\lambda)\coloneqq [F_{(w,\phi)}(\lambda,0,0)]\circ [\T(\lambda)]\colon Y\to X\]
whenever $\lambda\neq 0$. In order to simplify $\LL(\lambda)\theta$, we notice that
\[\A_w(\SL\theta)+\A_\phi(\psi_y^\lambda V[\SL\theta])=(\psi_y^\lambda-\lambda)V[S\theta].\]
Indeed, the function $f\coloneqq\A_w(\SL\theta)+\A_\phi(\psi_y^\lambda V[\SL\theta])-(\psi_y^\lambda-\lambda)V[S\theta]$ satisfies
\begin{align*}
	\Delta f&=-2\gamma(\psi^\lambda)\partial_yV[\SL\theta]-\gamma'(\psi^\lambda)\psi_y^\lambda V[\SL\theta]-\psi_{yyy}^\lambda V[\SL\theta]-2\psi_{yy}^\lambda\partial_y V[\SL\theta]\\
	&=-2\gamma(\psi^\lambda)\partial_yV[\SL\theta]-\gamma'(\psi^\lambda)\psi_y^\lambda V[\SL\theta]+\gamma'(\psi^\lambda)\psi_y^\lambda V[\SL\theta]+2\gamma(\psi^\lambda)\partial_y V[\SL\theta]=0
\end{align*}
in $\Omega_h$ and vanishes at the top and bottom. By \eqref{eq:Fder_trivial} and additionally using $\Ch\SL\theta_x=\SL\partial_y V[\SL\theta]$, we can thus write
\begin{subequations}\label{eq:Ltheta}
\begin{align}
	\LL_1(\lambda)\theta&=-\lambda^{-1}\SL\theta-\sigma^{-1}\partial_x^{-2}\left(\lambda\Pro\SL\partial_y(\A_\phi\theta+V[\SL\theta])+(\gamma(0)-\lambda^{-1}g)\SL\theta\right),\\
	\LL_2(\lambda)\theta&=\theta-(\A_\phi\theta+V[\SL\theta]).
\end{align}
\end{subequations}
Notice that, under the assumption $\theta\in C_\per^{2,\alpha}(\overline{\Omega_h})$, $\LL_2(\lambda)\theta$ is the unique solution of
\begin{align}\label{eq:L2theta_rewritten}
	\Delta[\LL_2(\lambda)\theta]=\Delta\theta+\gamma'(\psi^\lambda)\theta\text{ in }\Omega_h,\quad\LL_2(\lambda)\theta=0\text{ on }y=0\text{ and }y=-h,
\end{align}
and $\LL_1(\lambda)\theta$ is (in the set of periodic functions with zero average) uniquely determined by
\begin{align}\label{eq:L1theta_rewritten}
	[\LL_1(\lambda)\theta]_{xx}=-\lambda^{-1}\SL\theta_{xx}-\sigma^{-1}\left(\lambda\Pro\SL\partial_y(\theta-\LL_2(\lambda)\theta)+(\gamma(0)-\lambda^{-1}g)\SL\theta\right).
\end{align}
\subsection{Kernel}
Clearly, it suffices to study the kernel of $\LL$; here and in the following, we will suppress the dependency of $\LL$ on $\lambda$. From \eqref{eq:Ltheta}, \eqref{eq:L2theta_rewritten}, and \eqref{eq:L1theta_rewritten} we infer that
\begin{align*}
	\LL\theta=0\;\Longleftrightarrow\;\theta\in C_\per^{2,\alpha}(\overline{\Omega_h})\text{ and }\begin{cases}\Delta\theta+\gamma'(\psi^\lambda)\theta=0\quad\text{and}\\\lambda^{-1}\SL\theta_{xx}+\sigma^{-1}\left(\lambda\Pro\SL\partial_y\theta+(\gamma(0)-\lambda^{-1}g)\SL\theta\right)=0;\end{cases}
\end{align*}
notice that $\LL\theta=0$ implies $\SL\theta\in C^{2,\alpha}_{0,\per}(\R)$ and thus $\theta=\A_\phi\theta+V[\SL\theta]\in C_\per^{2,\alpha}(\overline{\Omega_h})$. Let us now write $\theta(x,y)=\sum_{k=0}^\infty\theta_k(y)\cos(k\nu x)$ as a Fourier series. Then we easily see that $\LL\theta=0$ if and only if \eqref{eq:LLtheta0=0} and for all $k\ge1$ \eqref{eq:LLthetak=0} holds, where
\begin{align}\label{eq:LLtheta0=0}
	\left(\partial_y^2+\gamma'(\psi^\lambda)\right)\theta_0=0,
\end{align}
noticing that $\theta_0(0)=0$ is already included in the definition of $Y$, and
\begin{subequations}\label{eq:LLthetak=0}
\begin{align}
	\left(\partial_y^2+\gamma'(\psi^\lambda)-(k\nu)^2\right)\theta_k&=0,\label{eq:LL1thetak=0}\\
	\left(-\sigma\lambda^{-1}(k\nu)^2+\gamma(0)-\lambda^{-1}g\right)\theta_k(0)+\lambda\partial_y\theta_k(0)&=0.
\end{align}
\end{subequations}
Henceforth, we shall assume that 
\begin{align}\label{ass:SL-spectrum}
	0\text{ is not in the Dirichlet spectrum of }\partial_y^2+\gamma'(\psi^\lambda)\text{ on }[-h,0],
\end{align}
Thus, we see that \eqref{eq:LLtheta0=0} only has the trivial solution $\theta_0=0$, recalling that $\theta_0(0)=\theta_0(-h)=0$ by membership of $\theta$ in $Y$.

Let us now turn to $k\ge 1$ and recall that $\theta_k(-h)=0$ by $\theta\in Y$. First suppose that $(k\nu)^2$ is in the Dirichlet spectrum of $\partial_y^2+\gamma'(\psi^\lambda)$ on $[-h,0]$. Then necessarily $\theta_k(0)=0$ provided \eqref{eq:LL1thetak=0} is satisfied. Hence, \eqref{eq:LLthetak=0} can only hold if $\theta_k=0$. On the other hand, suppose that $(k\nu)^2$ is not in the Dirichlet spectrum of $\partial_y^2+\gamma'(\psi^\lambda)$ on $[-h,0]$. Then we find that \eqref{eq:LL1thetak=0} is equivalent to the statement that $\theta_k$ is a multiple of $\beta^{-(k\nu)^2,\lambda}$, where $\beta=\beta^{\mu,\lambda}$ is defined to be the unique solution of the boundary value problem
\begin{align}\label{eq:beta}
	\beta_{yy}+(\gamma'(\psi^\lambda)+\mu)\beta=0\text{ on }[-h,0],\quad\beta(-h)=0,\quad\beta(0)=1.
\end{align}
Indeed, this problem has a unique solution for $\mu=-(k\nu)^2$, which can be seen as follows: Let $\tilde\beta$ solve the initial value problem
\[\tilde\beta_{yy}+(\gamma'(\psi^\lambda)+\mu)\tilde\beta=0\text{ on }[-h,0],\quad\tilde\beta(-h)=0,\quad\tilde\beta_y(-h)=1.\]
Then $\tilde\beta(0)\neq 0$ by the assumption on the Dirichlet spectrum. Therefore, $\beta$ solves \eqref{eq:beta} if and only if $\beta=\tilde\beta/\tilde\beta(0)$. Hence, \eqref{eq:LLthetak=0} has a nontrivial solution $\theta_k$ if and only if the dispersion relation
\[d(-(k\nu)^2,\lambda)=0,\]
where
\begin{align}\label{eq:d(k,lambda)}
	d(\mu,\lambda)\coloneqq\beta_y^{\mu,\lambda}(0)+\sigma\lambda^{-2}\mu+\lambda^{-1}\gamma(0)-\lambda^{-2}g,
\end{align}
is satisfied, and in this case $\theta_k$ is a multiple of $\beta^{-(k\nu)^2,\lambda}$.
\begin{remark}
	Clearly, $d(\mu,\lambda)$ is at first defined if and only if $\mu$ is not in the Dirichlet spectrum of $-\partial_y^2-\gamma'(\psi^\lambda)$ on $[-h,0]$. If this property fails to hold, we set $d(\mu,\lambda)\coloneqq\infty$ in the following.
\end{remark}
We summarise our results concerning the kernel:
\begin{lemma}\label{lma:kernel}
	Given $\lambda\neq 0$ and under the assumption \eqref{ass:SL-spectrum}, a function $\theta\in Y$, admitting the Fourier decomposition $\theta(x,y)=\sum_{k=0}^\infty\theta_k(y)\cos(k\nu x)$, is in the kernel of $\LL(\lambda)$ if and only if $\theta_0=0$ and for each $k\ge 1$
	\begin{enumerate}[label=(\alph*)]
		\item $\theta_k=0$, or
		\item $(k\nu)^2$ is not in the Dirichlet spectrum of $\partial_y^2+\gamma'(\psi^\lambda)$ on $[-h,0]$, $\theta_k$ is a multiple of $\beta^{-(k\nu)^2,\lambda}$, and the dispersion relation
		\[d(-(k\nu)^2,\lambda)=0\]
		holds, with $d$ given in \eqref{eq:d(k,lambda)}.
	\end{enumerate}
\end{lemma}

\subsection{Range}
Before we proceed with the investigation of the transversality condition, we first prove that the range of $\LL$ can be written as an orthogonal complement with respect to a suitable inner product. This will be helpful later. To this end, we introduce the inner product
\[\langle(f_1,g_1),(f_2,g_2)\rangle\coloneqq\langle f_1',f_2'\rangle_{L^2([0,L])}+\langle\nabla g_1,\nabla g_2\rangle_{L^2(\Omega_h^*)}\]
for $f_1,f_2\in H^1_{0,\per}(\R)$, $g_1,g_2\in H^1_\per(\Omega_h)$; in order to avoid misunderstanding, we point out that the index \enquote{0} in $H^1_{0,\per}(\R)$ means \enquote{zero average} as before and not \enquote{zero boundary values}. This inner product is positive definite on the space
\[H^1_{0,\per}(\R)\times\left\{\tilde g\in H^1_\per(\Omega_h):\tilde g\big|_{y=-h}=0\right\}.\]
Notice that
\[\langle f_1',f_2'\rangle_{L^2([0,L])}=-\langle f_1,f_2''\rangle_{L^2([0,L])}\]
if $f_2\in H^2_\per(\R)$ and that
\[\langle\nabla g_1,\nabla g_2\rangle_{L^2(\Omega_h^*)}=-\langle g_1,\Delta g_2\rangle_{L^2(\Omega_h^*)}+\langle\SL g_1,\SL\partial_y g_2\rangle_{L^2([0,L])}\]
if $g_2\in H^2_\per(\Omega_h)$ and $g_1=0$ on $y=-h$.

Using \eqref{eq:Ltheta} we now compute for smooth $\theta,\vartheta\in Y$
\begin{align*}
	&\langle(\sigma\lambda^{-1}\SL\theta,\theta),\LL\vartheta\rangle\\
	&=-\langle\theta,\Delta(\vartheta-(\A_\phi\vartheta+V[\SL\vartheta]))\rangle_{L^2(\Omega_h^*)}+\langle\SL\theta,\SL\partial_y(\vartheta-(\A_\phi\vartheta+V[\SL\vartheta]))\rangle_{L^2([0,L])}\\
	&\phantom{=\;}-\langle\SL\theta,-\sigma\lambda^{-2}\SL\vartheta_{xx}-\Pro\SL\partial_y(\A_\phi\vartheta+V[S\vartheta])-\lambda^{-1}(\gamma(0)-\lambda^{-1}g)\SL\vartheta\rangle_{L^2([0,L])}\\
	&=\langle\nabla\theta,\nabla\vartheta\rangle_{L^2(\Omega_h^*)}-\langle\theta,\gamma'(\psi^\lambda)\vartheta\rangle_{L^2(\Omega_h^*)}-\sigma\lambda^{-2}\langle\SL\theta_x,\SL\vartheta_x\rangle_{L^2([0,L])}\\
	&\phantom{=\;}-\lambda^{-1}(\gamma(0)-\lambda^{-1}g)\langle\SL\theta,\SL\vartheta\rangle_{L^2([0,L])}
\end{align*}
making use of $\langle\SL\theta\rangle=0$. Noticing that the terms at the beginning and at the end of this computation only involve at most first derivatives of $\theta$ and $\vartheta$, an easy approximation argument shows that this relation also holds for general $\theta,\vartheta\in Y$. Moreover, since the last expression is symmetric in $\theta$ and $\vartheta$, we can also go in the opposite direction with reversed roles and arrive at the symmetry property
\[\langle(\sigma\lambda^{-1}\SL\theta,\theta),\LL\vartheta\rangle=\langle\LL\theta,(\sigma\lambda^{-1}\SL\vartheta,\vartheta)\rangle.\]
Thus, the range of $\LL$ is the orthogonal complement of $\{(\sigma\lambda^{-1}\SL\theta,\theta):\theta\in\ker\LL\}$ with respect to $\langle\cdot,\cdot\rangle$. Indeed, one inclusion is an immediate consequence of the symmetry property and the other inclusion follows from the facts that we already know that $\LL$ is Fredholm with index zero and that $\LL$ gains no additional kernel when extended to functions $\theta$ of class $H^1$.

\subsection{Transversality condition}
Assuming that the kernel is spanned by $\theta(x,y)=\beta^{-(k\nu)^2,\lambda}(y)\cos(k\nu x)$, we have to investigate whether $\LL_\lambda\theta$ is not in the range of $\LL$, which is equivalent to
\[\langle(\sigma\lambda^{-1}\SL\theta,\theta),\LL_\lambda\theta\rangle\neq 0\]
by the preceding considerations. Differentiating \eqref{eq:Ltheta} with respect to $\lambda$, for general $\theta$ it holds that
\begin{align*}
	\LL_{\lambda,1}\theta&=\lambda^{-2}\SL\theta-\sigma^{-1}\partial_x^{-2}\left(\Pro\SL\partial_y(\A_\phi\theta+V[\SL\theta])+\lambda\Pro\SL\partial_y\A_{\phi\lambda}\theta+\lambda^{-2}g\SL\theta\right),\\
	\LL_{\lambda,2}\theta&=-\A_{\phi\lambda}\theta,
\end{align*}
where $\A_{\phi\lambda}\theta\in C_\per^{2,\alpha}(\overline{\Omega_h})$ is the unique solution of
\begin{align*}
	\Delta(\A_{\phi\lambda}\theta)&=-\gamma''(\psi^\lambda)\partial_\lambda\psi^\lambda\theta&\text{in }\Omega_h,\\
	\A_{\phi\lambda}\theta&=0&\text{on }y=0\text{ and }y=-h.
\end{align*}
Thus, we have
\begin{align*}
	&\langle(\sigma\lambda^{-1}\SL\theta,\theta),\LL_\lambda\theta\rangle\\
	&=-\langle\theta,\Delta(-\A_{\phi\lambda}\theta)\rangle_{L^2(\Omega_h^*)}+\langle\SL\theta,-\SL\partial_y\A_{\phi\lambda}\theta\rangle_{L^2([0,L])}\\
	&\phantom{=\;}-\langle\sigma\lambda^{-1}\SL\theta,\lambda^{-2}\SL\theta_{xx}-\sigma^{-1}\left(\Pro\SL\partial_y(\A_\phi\theta+V[\SL\theta])+\lambda\Pro\SL\partial_y\A_{\phi\lambda}\theta+\lambda^{-2}g\SL\theta\right)\rangle_{L^2([0,L])}\\
	&=-\langle\theta,\gamma''(\psi^\lambda)\partial_\lambda\psi^\lambda\theta\rangle_{L^2(\Omega_h^*)}-\langle\SL\theta,2\sigma\lambda^{-3}\SL\theta_{xx}+(\lambda^{-2}\gamma(0)-2\lambda^{-3}g)\SL\theta\rangle_{L^2([0,L])}
\end{align*}
whenever $\LL_1\theta=0$. Now let $\theta(x,y)=\beta^{-(k\nu)^2,\lambda}(y)\cos(k\nu x)$ and notice that $f=\partial_\lambda\beta^{-(k\nu)^2,\lambda}$ solves
\[f_{yy}+(\gamma'(\psi^\lambda)-(k\nu)^2)f=-\gamma''(\psi^\lambda)\beta^{-(k\nu)^2,\lambda}\partial_\lambda\psi^\lambda\text{ on }[-h,0],\quad f(0)=f(-h)=0.\]
Therefore,
\begin{align*}
	&\frac2L\langle(\sigma\lambda^{-1}\SL\theta,\theta),\LL_\lambda\theta\rangle\\
	&=\int_{-h}^0\beta^{-(k\nu)^2,\lambda}\left(\partial_\lambda\beta^{-(k\nu)^2,\lambda}_{yy}+(\gamma'(\psi^\lambda)-(k\nu)^2)\partial_\lambda\beta^{-(k\nu)^2,\lambda}\right)\,dy+2\sigma\lambda^{-3}(k\nu)^2\\
	&\phantom{=\;}-\lambda^{-2}\gamma(0)+2\lambda^{-3}g\\
	&=\partial_\lambda\beta_y^{-(k\nu)^2,\lambda}(0)+2\sigma\lambda^{-3}(k\nu)^2-\lambda^{-2}\gamma(0)+2\lambda^{-3}g
\end{align*}
after integrating by parts. Thus, we have proved:
\begin{lemma}\label{lma:transversality_condition}
	Given $\lambda\neq 0$ and assuming that the kernel of $\LL(\lambda)$ is one-dimensional spanned by $\theta(x,y)=\beta^{-(k\nu)^2,\lambda}(y)\cos(k\nu x)$ for some $k\ge 1$, the transversality condition
	\[\LL_\lambda(\lambda)\theta\notin\im\LL(\lambda)\]
	is equivalent to
	\[d_\lambda(-(k\nu)^2,\lambda)\neq 0,\]
	with $d$ given in \eqref{eq:d(k,lambda)}.
\end{lemma}

\subsection{Result on local bifurcation}
We summarise our result of this section applying Theorem \ref{thm:CrandallRabinowitz}.
\begin{theorem}\label{thm:LocalBifurcation}
	Assume that there exists $\lambda_0\neq 0$ such that \eqref{ass:SL-spectrum} holds for $\lambda=\lambda_0$ and such that the dispersion relation
	\[d(-(k\nu)^2,\lambda_0)=0,\]
	with $d$ given by \eqref{eq:d(k,lambda)}, has exactly one solution $k_0\in\N$, and assume that the transversality condition
	\[d_\lambda(-(k_0\nu)^2,\lambda_0)\neq 0\]
	holds. Then there exists $\varepsilon>0$ and a $C^1$-curve $(-\varepsilon,\varepsilon)\ni s\mapsto(\lambda(s),w(s),\phi(s))$ with $(\lambda(0),w(0),\phi(0))=(\lambda_0,0,0)$, $w(s)\neq 0$ for $s\neq 0$, and $F(\lambda(s),w(s),\phi(s))=0$. Moreover, all solutions of $F(\lambda,w,\phi)=0$ in a neighbourhood of $(\lambda_0,0,0)$ are on this curve or are trivial. Furthermore, the curve admits the asymptotic expansion $(w(s),\phi(s))=s\T(\lambda_0)\theta+o(s)$, where
	\begin{align*}
		\theta(x,y)&=\beta^{-(k_0\nu)^2,\lambda_0}(y)\cos(k_0\nu x),\\
		[\T(\lambda_0)\theta](x,y)&=\left(-\frac{1}{\lambda_0},\beta^{-(k_0\nu)^2,\lambda_0}(y)-\frac{\psi_y^{\lambda_0}(y)\sinh(k_0\nu(y+h))}{\lambda_0\sinh(k_0\nu h)}\right)\cos(k_0\nu x).
	\end{align*}
	These solutions give rise to proper water waves, that is, additionally \eqref{eq:additional_requirements} is satisfied.
\end{theorem}
\begin{proof}
	It is straightforward to apply Theorem \ref{thm:CrandallRabinowitz} in view of Lemmas \ref{lma:M_prop}, \ref{lma:kernel}, and \ref{lma:transversality_condition}, since $F_{(w,\phi)}(\lambda_0,0,0)$ coincides with $\LL(\lambda_0)$ up to the isomorphism $\T(\lambda_0)$. Moreover, the asymptotic expansion tells us that $w(s)\neq 0$ after possibly shrinking $\varepsilon$. Furthermore, after again possibly shrinking $\varepsilon$, we see that also \eqref{eq:additional_requirements} is satisfied along the curve.
\end{proof}

\section{Conditions for local bifurcation}
In this section we further study the conditions on local bifurcation in Theorem \ref{thm:LocalBifurcation}.

In general, one cannot exclude the possibility of kernels of dimension larger than one. In some cases, which we will study below, the kernel can be at most two-dimensional or is exactly one-dimensional. However, in general kernels of arbitrarily large dimension should be expected to appear; we refer at this point to the discussions in \cite{AasVar18,EHR09} for the case without surface tension.

\subsection{General conclusions}\label{sec:general_conclusions}
In order to put hands on $\beta^{\mu,\lambda}_y(0)$, the usual strategy \cite{BBCW04,Varholm20} is to introduce the Prüfer angle $\vartheta=\vartheta(y,\mu)$ corresponding to the initial value problem
\begin{gather*}
-u''-\gamma'(\psi^\lambda)u=\mu u\text{ on }[-h,0],\\
u(-h)=0,\quad u'(-h)=1,
\end{gather*}
as the unique continuous representative of $\mathrm{arg}(u'+iu)$ with $\vartheta(-h,\mu)=\mathrm{arg}(1)=0$. Thus,
\[\beta^{\mu,\lambda}_y(0)=\cot(\vartheta(0,\mu)),\]
and one is, for fixed $\lambda$, left to search for intersections of this quantity with the straight line
\begin{align*}
	A(\mu)\coloneqq-\sigma\lambda^{-2}\mu-\lambda^{-1}\gamma(0)+\lambda^{-2}g
\end{align*}
when studying the dispersion relation. This is the same as to search for eigenvalues $\mu$ of the following Sturm--Liouville problem with eigenvalue-dependent boundary condition:
\begin{subequations}\label{eq:spectral}
	\begin{align}
		-\beta''-\gamma'(\psi^\lambda)\beta&=\mu \beta\quad\text{ on }[-h,0],\\
		\beta(-h)&=0,\\
		-\sigma^{-1}\lambda^2 \beta'(0)+\sigma^{-1}\lambda^2 (\lambda^{-2}g -\lambda^{-1}\gamma(0))\beta(0)&=\mu\beta(0).
	\end{align}
\end{subequations}
Now \cite[Thm. 2.1.]{BBCW04} yields the following result.
\begin{proposition}\label{thm:branches}
	For fixed $\lambda\ne0$, the problem \eqref{eq:spectral} has discrete spectrum, bounded from below, and accumulating only at $+\infty$. Moreover, one of the following alternatives occur:
	\begin{enumerate}[label=(\roman*)]
		\item All eigenvalues are real, there are algebraically two eigenvalues on the initial
branch, $\mathcal B_0$, of the Prüfer graph, and all other branches contain precisely one
simple eigenvalue.
		\item All eigenvalues are real, $\mathcal B_0$ contains no eigenvalues, but, for some $k>0$, $\mathcal B_k$ contains algebraically three eigenvalues and all other branches contain
precisely one simple eigenvalue.
		\item There are two non-real eigenvalues appearing as a conjugate pair, $\mathcal B_0$ contains no eigenvalues, and all other branches contain precisely one simple eigenvalue.
	\end{enumerate}
	Here, $\mathcal B_0$, $\mathcal B_1$, $\ldots$ denote the branches of the graph of $\mu\mapsto\cot(\vartheta(0,\mu))=\beta^{\mu,\lambda}_y(0)$, listed from left to right.
\end{proposition}
In \cite{Varholm20} (using $z=-\mu$ instead of $\mu$ and with the length of the interval scaled to $1$) it was proved that $\beta^{\mu,\lambda}_y(0)$ is strictly monotonically decreasing in $\mu$ on each branch and satisfies the bounds
\begin{align}\label{eq:bounds_beta}
	v(\mu+\sup\gamma')\le\beta^{\mu,\lambda}_y(0)\le v(\mu+\inf\gamma')
\end{align}
with
\[v(z)\coloneqq\sqrt{-z}\coth(h\sqrt{-z})\]
on the (possibly empty) intervals
\[I_j=\begin{cases}(j^2\pi^2/h^2-\inf\gamma',(j+1)^2\pi^2/h^2-\sup\gamma'),&j\in\N,\\(-\infty,\pi^2/h^2-\sup\gamma'),&j=0.\end{cases}\]
With this, the following theorem can be proved.
\begin{proposition}\label{thm:large lambda}
	For $|\lambda|$ sufficiently large, there is at least one solution $\mu$ of $d(\mu,\lambda)=0$ on the leftmost branch of $\mu\mapsto\beta^{\mu,\lambda}_y(0)$, and the lowest such solution tends to $-\infty$ as $|\lambda|\to\infty$.
\end{proposition}
\begin{proof}
	First we note that the slope of $A$ and the value of $A$ at $\mu=0$ approaches $0$ as $|\lambda|\to\infty$. Moreover, on the leftmost branch, $\beta^{\mu,\lambda}_y(0)$ is strictly monotonically decreasing, converges to $-\infty$ as $\mu$ approaches the first Dirichlet eigenvalue of $-\partial_y^2-\gamma'(\psi^\lambda)$, and satisfies the in $\lambda$ uniform bounds \eqref{eq:bounds_beta}, which in particular imply that it behaves like $\sqrt{-\mu}$ for $-\mu$ large. From these observations the first assertion follows for $|\lambda|$ large enough such that $A(\mu_0)<1$, where $\mu_0<-\sup\gamma'$ is such that $v(\mu_0+\sup\gamma')=1$, and the second assertion follows easily.
\end{proof}

Now we want to say something more about the conditions for local bifurcation imposing sign conditions on $\gamma'$ and/or $\gamma''$. 
In fact, it will be sufficient to assume that $\gamma'$ is \enquote*{not too positive}.

\paragraph{Dispersion relation}

Note that in the irrotational setting, the kernel is at most two-dimensional; see \cite{Jones89}. We begin by generalising this result.
\begin{proposition}
	If $\sup \gamma'\le \frac{\pi^2}{h^2}$, then there are at most two non-positive solutions $\mu$ to the equation $d(\mu, \lambda)=0$.
\end{proposition}

\begin{proof}
	It follows from Proposition \ref{thm:branches} that there are at most two solutions on the leftmost branch of $\mu \mapsto \beta_y^{\mu, \lambda}(0)$ and by assumption and \eqref{eq:bounds_beta} this branch has non-negative right endpoint.
\end{proof}

\begin{remark}\label{rem:two_sol}
	If there appears the case that $d(-(k\nu)^2,\lambda)=0$ has exactly two different solutions $k_1,k_2\in\N$, $k_1<k_2$, one can handle this problem as follows: If $k_2/k_1\notin\N$, then one can consider, from the beginning, only functions with period $L/k_1$ and $L/k_2$, respectively, to obtain two different solution curves made up of functions with period $L/k_1$ and $L/k_2$, respectively. If, on the other hand, $k_2/k_1\in\N$, then this procedure can only be done in the space of functions with period $L/k_2$, yielding only one corresponding curve. Similar arguments apply in the case of higher-dimensional kernels (which might appear if $\sup\gamma'>\frac{\pi^2}{h^2}$).
\end{remark}

It is also of interest to know when we can expect to find a single negative solution. In the irrotational case, a well known criterion is that the Froude number is greater than one, or equivalently, $\lambda^2>gh$. This can also be generalised.

\begin{proposition}
	Assume that $\sup \gamma'< \frac{\pi^2}{h^2}$ and that 
	\[
	\lambda^{-2}g -\lambda^{-1}\gamma(0)< 
	\begin{cases} 
	\sqrt{\sup \gamma'} \cot (h\sqrt{\sup \gamma'}) & \text{if } \sup \gamma'>0,\\
	\frac{1}{h} & \text{if } \sup \gamma'=0,\\
	\sqrt{|\sup \gamma'|} \coth (h\sqrt{|\sup \gamma'}|) & \text{if } \sup \gamma'<0.
	\end{cases}
	\]
	Then there is precisely one negative solution $\mu_0$ to the equation $d(\mu, \lambda)=0$. In particular,
	\begin{enumerate}[label=(\roman*)]
		\item the dispersion relation $d(-(k\nu)^2,\lambda)=0$ has in $\N$ the unique solution $k_0$ if $k_0\in\N$ and $\nu>0$ are chosen such that $-(k_0\nu)^2=\mu_0$;
		\item if $\sup \gamma' < \frac{\pi^2}{4h^2}$, this criterion is satisfied for sufficiently large $|\lambda|$.
	\end{enumerate}
\end{proposition}

\begin{proof}
	The first assumption guarantees that the leftmost branch of $\mu \mapsto \beta_y^{\mu, \lambda}(0)$ has  positive right endpoint, and the second that $A(0)<\beta_y^{0, \lambda}(0)$. Arguing as in the proof of Proposition \ref{thm:large lambda} we find that there must be one intersection point between $A(\mu)$ and $\beta_y^{\mu, \lambda}(0)$ to the left of the origin and one to the right. Finally, if $\sup \gamma'<\frac{\pi^2}{4h^2}$ then $\beta_y^{0, \lambda}(0)$ is necessarily positive, so that the second condition is satisfied for large $|\lambda|$.
\end{proof}

Note that this reduces to the condition $\lambda^2>gh$ in the irrotational case and to the supercriticality condition
$\lambda^{-2}g -\lambda^{-1} \gamma < 1/h$ in the case of constant vorticity $\gamma$.

\paragraph{Transversality condition}

We next try to find a pretty general sufficient condition for the transversality condition. For this, we will make use of the following general lemma.
\begin{lemma}\label{lma:max_principle}
	Let $c\in C([-h,0])$ with $\sup c<\frac{\pi^2}{4h^2}$, $u\in C^2((-h,0))\cap C([-h,0])$, and $Lu\coloneqq-u''-cu$. We have:
	\begin{enumerate}[label=(\roman*)]
		\item If $Lu\le0$ on $(-h,0)$, $u(-h)=0$, and $u(0)>0$, then $u'(0)>0$.
		\item If $Lu\ge0$ on $(-h,0)$, $u(-h)=0$, and $u(0)\ge0$, then $u\ge0$ on $[-h,0]$.
		\item If $Lu\ge0$ on $(-h,0)$, $u(0)=0$, and $u'(0)>0$, then $u<0$ on $[-h,0)$.
	\end{enumerate}
\end{lemma}
\begin{proof}
	Let $y_0>0$ be sufficiently small so that $\sup c < \frac{\pi^2}{4(h+y_0)^2}$  and choose  $q>0$ so that
	\[
	\sup c < q^2< \frac{\pi^2}{4(h+y_0)^2}.
	\]
	Then the function $w(y)=\cos (q(y-y_0))$ satisfies $w(y)>0$ and $w'(y)>0$ on $[-h,0]$, and
	\[
	-w''-q^2 w=0.
	\]
	Moreover, an easy computation shows that $v\coloneqq u/w$ satisfies
	\[Mv\coloneqq-v''-\frac{2w'}{w}v'+(q^2-c)v=\frac{Lu}{w}\quad\mathrm{on\ }(-h,0).\]
	Now, to this equation we can apply standard maximum principles as $q^2-c\ge0$.
	
	In the situation of (i), we have $Mv\le0$ on $(-h,0)$, $v(-h)=0$, and $v(0)>0$. By the weak maximum principle, it follows that $v(0)=\max v$. Therefore, $v'(0)\ge0$ and consequently $u'(0)=v'(0)w(0)+v(0)w'(0)>0$.
	
	In the situation of (ii), we have $Mv\ge0$ on $(-h,0)$, $v(-h)=0$, and $v(0)\ge0$. By the weak maximum principle, it follows that $v\ge0$ and thus $u\ge0$ on $[-h,0]$.
	
	In the situation of (iii), we have $Mv\ge0$ on $(-h,0)$, $v(0)=0$, and $v'(0)=u'(0)/w(0)>0$. By the strong maximum principle, it follows that $v$ cannot have an interior minimum point on any $(a,0)$, $-h\le a<0$. Hence, $v$ is monotonically increasing on $[-h,0]$ because of $v'(0)>0$. In particular, $v<v(0)=0$ and thus $u<0$ on $[-h,0)$.
\end{proof}
Now let us assume that $\sup \gamma' < \frac{\pi^2}{4h^2}$, and moreover $\mu\le 0$. Observe that $\beta=\beta^{\mu,\lambda}$ satisfies
\[
-\beta''-(\gamma'(\psi^\lambda)+\mu)\beta=0, \quad \beta(-h)=0, \quad \beta(0)=1.
\]
By Lemma \ref{lma:max_principle}(i),(ii), it follows that $\beta'(0)>0$ and $\beta\ge0$ on $[-h,0]$.
Moreover, since
\[
-(\partial_\lambda \psi^\lambda)''-\gamma'(\psi^\lambda)\partial_\lambda \psi^\lambda=0,\quad \partial_\lambda \psi^\lambda(0)=0, \quad (\partial_\lambda \psi^\lambda)'(0)=1,
\]
we find that $\partial_\lambda \psi^\lambda<0$ on $[-h, 0)$ applying Lemma \ref{lma:max_principle}(iii).
Finally, note that $\partial_\lambda \beta$ satisfies 
\[
-\partial_\lambda \beta''-(\gamma'(\psi^\lambda)+\mu)\partial_\lambda \beta=\gamma''(\psi^\lambda)\partial_\lambda \psi^\lambda \beta, \quad 
\partial_\lambda \beta(-h)=0, \quad \partial_\lambda \beta(0)=0.
\]
Assuming that $\gamma''\le 0$, the right hand side is non-negative and  we can apply Lemma \ref{lma:max_principle}(ii) to conclude that $\partial_\lambda \beta\ge 0$ on $[-h, 0]$ and therefore $(\partial_\lambda \beta)'(0) \le 0$. 
Next note that
\[
\gamma(0)=-\lambda \beta'(0)+\lambda^{-1}(-\sigma\mu+g)
\]
if $d(\mu,\lambda)=0$ holds. Substituting this into the transversality condition, we get
\begin{align*}
d_\lambda(\mu,\lambda)&=\partial_\lambda \beta'(0) +2\lambda^{-3}(-\sigma\mu+g)-\lambda^{-2} \gamma(0)\\
&=\partial_\lambda \beta'(0) +\lambda^{-1}\beta'(0)+\lambda^{-3}(-\sigma\mu+g)<0
\end{align*}
if $\lambda<0$. On the other hand, if $\gamma'' \ge 0$ and $\lambda>0$, we get by the same argument that $d_\lambda(\mu,\lambda)>0$.
In the special case of affine vorticity, we have $\gamma''\equiv 0$, and so the argument works irrespective of the sign of $\lambda$.

To summarise, we have proved the following.
\begin{theorem}
	Assume $\sup \gamma'< \frac{\pi^2}{4h^2}$, $\mu\le0$, and
	\begin{enumerate}[label=(\roman*)]
		\item $\gamma''\le0$ and $\lambda<0$, or
		\item $\gamma''\ge0$ and $\lambda>0$.
	\end{enumerate}
	Then $d_\lambda(\mu,\lambda)\ne0$ provided $d(\mu,\lambda)=0$.
\end{theorem}

It should be noted that the condition $\sup\gamma'< \frac{\pi^2}{4h^2}$ still allows for waves with critical layers. If we, for example, take 
\[
\gamma(\psi)=a\psi+b
\]
with $a<0$ and $b>0$, then
\[
\psi_y^\lambda(y)=\lambda\cosh(\sqrt{-a}y)-\frac{b}{\sqrt{-a}}\sinh(\sqrt{-a}y),
\]
which changes sign if 
\[
-\frac{b}{\sqrt{-a}} \tanh(\sqrt{-a}h)<\lambda<0.
\]

\subsection{Case of unidirectional flows}\label{sec:ConditionsLocalBifurcation_unidirectional}
We now want to study further the conditions for local bifurcation under the assumption that the trivial stream function $\psi^\lambda$ is a strictly decreasing function of $y$; thus, necessarily $\lambda<0$. In this case the change of variables $p=-\psi^\lambda(y)$ yields the differential equation
\[G_{pp}=-\gamma(-p)G_p^3\text{ on }[p_0,0]\]
for $G(p)=y+h$ corresponding to $\psi^\lambda$; here, $p_0=m(\lambda)<0$. Moreover, the boundary conditions for $\psi^\lambda$ translate to boundary conditions for $G$:
\[G(p_0)=0,\quad G_p(0)=-1/\lambda.\]
Thus,
\begin{align*}
G(p)=\int_{p_0}^p\frac{ds}{\sqrt{\lambda^2+2\Gamma(s)}},\quad\text{where }\Gamma(p)\coloneqq\int_0^p\gamma(-s)\,ds,\quad p_0\le p\le 0,
\end{align*}
provided $\lambda^2>-2\min_{[p_0,0]}\Gamma$. Let us denote $a\coloneqq\sqrt{\lambda^2+2\Gamma}=1/G_p$; in particular, $a(0)=-\lambda$.

We now prove that, when considering the dispersion relation, a certain well-known Sturm--Liouville problem can be investigated equivalently; this Sturm--Liouville problem, in turn, was studied in detail in \cite{Wahlen06b}.
\begin{proposition}
	Let $\lambda<0$, $\mu\in\R$, and $p_0\coloneqq m(\lambda)$. Suppose that $\psi^\lambda$ is strictly decreasing on $[-h,0]$ and that $\lambda^2>-2\min_{[p_0,0]}\Gamma$. Then, $d(\mu,\lambda)=0$ if and only if the Sturm--Liouville problem\begin{subequations}\label{eq:SturmLiouville}
		\begin{align}
		(a^3v_p)_p&=-\mu av\quad\mathrm{on\ }[p_0,0],\label{eq:SturmLiouville1}\\
		a^3(0)v_p(0)&=(-\sigma\mu+g)v(0),\label{eq:SturmLiouville2}\\
		v(p_0)&=0,\label{eq:SturmLiouville3}
		\end{align}	
	\end{subequations}
	has a nontrivial solution.
\end{proposition}
\begin{proof}
	Let us first assume that $\mu$ is not in the Dirichlet spectrum of $-\partial_y^2-\gamma'(\psi^\lambda)$ on $[-h,0]$. Let $v\coloneqq BG_p$, where $B(p)=\beta(y)$, $\beta=\beta^{\mu,\lambda}$ introduced in \eqref{eq:beta}, and note that $v\not\equiv 0$. We have $v(p_0)=0$ and (by abuse of notation)
	\begin{align*}
	v_p=B_pG_p+BG_{pp}=\beta_yG_p^2-B\gamma(-p)G_p^3.
	\end{align*}
	Thus,
	\begin{align*}
	(a^3v_p)_p&=\left(\frac{\beta_y}{G_p}-B\gamma(-p)\right)_p=\beta_{yy}-\frac{\beta_yG_{pp}}{G_p^2}-\beta_yG_p\gamma(-p)+B\gamma'(-p)\\
	&=-(\gamma'(-p)+\mu)B+\beta_y\gamma(-p)G_p-\beta_yG_p\gamma(-p)+B\gamma'(-p)=-\mu B\\
	&=-\mu av.
	\end{align*}
	Therefore, the solutions of \eqref{eq:SturmLiouville1}, \eqref{eq:SturmLiouville3} are exactly the multiples of $v=BG_p$. For this $v$, we further have
	\[a^3(0)v_p(0)-(-\sigma\mu+g)v(0)=-\lambda^3\left(\frac{\beta_y(0)}{\lambda^2}+\frac{\gamma(0)}{\lambda^3}\right)+\frac{-\sigma\mu+g}{\lambda}=-\lambda d(\mu,\lambda),\]
	from which the assertion follows immediately.
	
	If, on the other hand, $\mu$ is in the Dirichlet spectrum of $-\partial_y^2-\gamma'(\psi^\lambda)$ on $[-h,0]$, then the same computation as above shows that the solutions of \eqref{eq:SturmLiouville1}, \eqref{eq:SturmLiouville3} are exactly the multiples of $\tilde v=\tilde BG_p$, where $\tilde B(p)=\tilde\beta(y)$ and $\tilde\beta$ is determined by
	\[\tilde\beta_{yy}+(\gamma'(\psi^\lambda)+\mu)\tilde\beta=0\text{ on }[-h,0],\quad\tilde\beta(-h)=0,\quad\tilde\beta_y(-h)=1.\]
	Then, however, $\tilde v(0)=0$ as $\tilde\beta(0)=0$ by hypothesis and thus $\tilde v\equiv 0$ whenever \eqref{eq:SturmLiouville2} is satisfied.
\end{proof}
We moreover can prove the following.
\begin{proposition}
	Assume that there exists $0\le\delta<1$ and $C>0$ such that
	\[|\gamma(p)|\le C(|p|^\delta+1),\quad p>0,\]
	and let $\lambda<0$. Then we have:
	\begin{enumerate}[label=(\alph*)]
		\item It holds that $\lambda^2>-2\min_{[p_0,0]}\Gamma$ if $|\lambda|$ is sufficiently large.
		\item If $\lambda^2>-2\min_{[p_0,0]}\Gamma$, there are at most two $\mu<0$ such that \eqref{eq:SturmLiouville} has a nontrivial solution (or, equivalently, $d(\mu,\lambda)=0$).
		\item There is exactly one $\mu_0<0$ such that \eqref{eq:SturmLiouville} has a nontrivial solution (or, equivalently, $d(\mu_0,\lambda)=0$) if $|\lambda|$ is sufficiently large. For such $\lambda$, in particular, the dispersion relation $d(-(k\nu)^2,\lambda)=0$ has in $\N$ the unique solution $k_0$ if $k_0\in\N$ and $\nu>0$ are chosen such that $-(k_0\nu)^2=\mu_0$.
	\end{enumerate}
\end{proposition}
\begin{proof}
	In the following, $c>0$ denotes some generic constant which may depend on $\delta$, $C$, and $h$ and can change from line to line. Recalling \eqref{eq:trivial_sol}, we see that
	\begin{align*}
		\frac{d}{dy}\frac12\left((\psi^\lambda)^2+(\psi^\lambda_y)^2\right)&=\psi^\lambda\psi^\lambda_y-\psi^\lambda_y\gamma(\psi^\lambda)\le\psi^\lambda\psi^\lambda_y+C\left(|\psi^\lambda|^\delta+1\right)|\psi^\lambda_y|\\
		&\le c\left((\psi^\lambda)^2+(\psi^\lambda_y)^2+1\right).
	\end{align*}
	Therefore,
	\[(\psi^\lambda)^2+(\psi^\lambda_y)^2\le c(\lambda^2+1)\]
	by Gronwall's inequality and, recalling \eqref{eq:def_m},
	\begin{align}\label{eq:est_p0}
		|p_0|=|m(\lambda)|\le c\sqrt{\lambda^2+1}\le c(|\lambda|+1)
	\end{align}
	in particular. Thus, we have
	\begin{align}\label{eq:est_Gamma}
		\Gamma(p)\ge-\int_{p_0}^0C(|s|^\delta+1)\,ds\ge -C|p_0|(|p_0|^\delta+1)\ge -c\left(|\lambda|^{\delta+1}+1\right),\quad p_0\le p\le 0,
	\end{align}
	which proves (a) in view of $2>\delta+1$.
	
	As for (b) and (c), we follow the idea of \cite{Wahlen06b} (pointing out the fundamental difference to that paper since in our case $p_0$ is not fixed and depends on $\lambda$) and introduce the complex Pontryagin space $\mathbb H=L^2([p_0,0])\times\C$, which becomes, equipped with the inner product
	\[[\tilde u_1,\tilde u_2]=\langle au_1,u_2\rangle_{L^2([p_0,0])}-\sigma b_1\overline{b_2},\]
	a $\pi_1$-space; here and it what follows, we write $\tilde u=(u,b)\in\mathbb H$. After defining the linear operator $T$ via
	\[T\tilde u=\left(-a^{-1}(a^3u')',-a^3(0)\sigma^{-1}u'(0)+g\sigma^{-1}u(0)\right)\]
	with domain
	\[D(T)=\left\{(u,b)\in\mathbb H:u\in H^2((p_0,0)),u(p_0)=0,b=u(0)\right\},\]
	\eqref{eq:SturmLiouville} is cast into the eigenvalue problem $T\tilde u=\mu\tilde u$. In \cite[Lem. 3.8., Lem. 3.9.]{Wahlen06b} (and their proofs) it was shown that $T$ has at most two negative eigenvalues, which implies (b), and exactly one negative eigenvalue if
	\begin{align}\label{eq:Wahlen_cond_onedim}
		\int_{p_0}^0a^{-3}(p)\,dp<\frac1g.
	\end{align}
	Now
	\[a=\sqrt{\lambda^2+2\Gamma}\ge\sqrt{\lambda^2-c\left(|\lambda|^{\delta+1}+1\right)}\]
	by \eqref{eq:est_Gamma}, at least for $|\lambda|$ sufficiently large. Thus and by \eqref{eq:est_p0}, we have
	\[\int_{p_0}^0a^{-3}(p)\,dp\le c(|\lambda|+1)\left(\lambda^2-c\left(|\lambda|^{\delta+1}+1\right)\right)^{-3/2}\overset{|\lambda|\to\infty}{\longrightarrow}0.\]
	In particular, \eqref{eq:Wahlen_cond_onedim} is satisfied for sufficiently large $|\lambda|$, which completes the proof of (c).
\end{proof}

\subsection{Examples}\label{sec:ConditionsLocalBifurcation_examples}
We illustrate our results with certain examples.
\subsubsection{Constant vorticity}
First we suppose that $\gamma$ is a constant. Then,
\[\psi^\lambda(y)=-\frac\gamma2 y^2+\lambda y,\]
so $\psi_y^\lambda<0$ on $[-h,0]$ if and only if
\begin{align}\label{eq:const_vort_lambda_cond_mon}
	\lambda<\begin{cases}0,&\gamma\le 0,\\-\gamma h,&\gamma>0.\end{cases}
\end{align}
Clearly, $(k\nu)^2$ is not in the Dirichlet spectrum of $\partial_y^2$ on $[-h,0]$ for any $k\ge 0$, and it holds that
\[\beta^{-(k\nu)^2,\lambda}(y)=\frac{\sinh(k\nu(y+h))}{\sinh(k\nu h)}.\]
Thus, we have
\[d(-(k\nu)^2,\lambda)=k\nu\coth(k\nu h)-\sigma\lambda^{-2}(k\nu)^2+\lambda^{-1}\gamma-\lambda^{-2}g,\]
and $d(-(k\nu)^2,\lambda)=0$ can equivalently be written as	
\begin{align*}
	\lambda&=\frac{-\gamma\pm\sqrt{\gamma^2+4(\sigma l^2+g)l\coth(lh)}}{2l\coth(lh)}\\
	&=-\frac{\gamma\tanh(lh)}{2l}\pm\sqrt{\frac{\sigma l^2+g}{l}\tanh(lh)+\frac{\gamma^2}{4l^2}\tanh^2(lh)}\eqqcolon\lambda^\pm(l)
\end{align*}
with $l=k\nu$, a formula which can also be found in \cite{Martin13,Wahlen06b}. Notice that we always have $d_\lambda(-(k\nu)^2,\lambda)\neq 0$---also if \eqref{eq:const_vort_lambda_cond_mon} is violated---since the discriminant of $d(-(k\nu)^2,\lambda)=0$, viewed as a quadratic equation in $\lambda^{-1}$, is
\[\gamma^2+4(\sigma(k\nu)^2+g)k\coth(k\nu h)>0.\]
For a more detailed analysis of roots of the dispersion relation, we refer to \cite{Martin13}. Here, we like to mention the characterisation of the cases when the dispersion relation has only the solution $k=1$ at $\lambda=\lambda^\pm(\nu)$, giving rise to an exactly one-dimensional kernel (generated by a function with minimal period $L$) at $\lambda=\lambda^\pm(\nu)$: These are exactly the cases where
\[\frac{\sigma}{gh^2}\ge\frac{\gamma^2h}{6g}+\frac13\mp\frac{\gamma}{6g}\sqrt{\gamma^2h^2+4gh}\]
or
\[\frac{\sigma}{gh^2}<\frac{\gamma^2h}{6g}+\frac13\mp\frac{\gamma}{6g}\sqrt{\gamma^2h^2+4gh}\quad\text{and}\quad\lambda^\pm(\nu)\ne\lambda^\pm(k\nu)\text{ for all }k\ge 2.\]
(The statements for the \enquote{$+$} and \enquote{$-$} situation are to be read separately.)

\subsubsection{Affine vorticity}
If the vorticity is affine, that is, $\gamma(\psi)=a\psi+b$ with $a\ne 0$, then the trivial solutions are given by
\[\psi^\lambda(y)=\begin{cases}\frac{\lambda}{\sqrt{a}}\sin(\sqrt{a}y)+\frac{b}{a}(\cos(\sqrt{a}y)-1),&a>0,\\\frac{\lambda}{\sqrt{-a}}\sinh(\sqrt{-a}y)+\frac{b}{a}(\cosh(\sqrt{-a}y)-1),&a<0.\end{cases}\]
In the case $a<0$, it holds that $\psi_y^\lambda<0$ on $[-h,0]$ if and only if
\[\lambda<\begin{cases}0,&b\le 0,\\-\frac{b}{\sqrt{-a}}\tanh(\sqrt{-a}h),&b>0,\end{cases}\]
and in the case $a>0$ if and only if
\[\lambda<0\quad\text{and}\quad\sqrt{a}h<\cot^{-1}\left(-\frac{b}{\lambda\sqrt{a}}\right),\]
where $\cot^{-1}\colon\R\to(0,\pi)$. But in general, these trivial solutions can have arbitrarily many critical layers if $a>0$. A short computation shows that $(k\nu)^2$ is in the Dirichlet spectrum of $\partial_y^2+a$ on $[-h,0]$ if and only if
\begin{align}\label{eq:affine_vort_Dir_spec}
	\exists m\in\N:a-(k\nu)^2=\frac{\pi^2}{h^2}m^2.
\end{align}
To compute $\beta^{-(k\nu)^2,\lambda}$ and $d(-(k\nu)^2,\lambda)$, we have to distinguish three cases:
\begin{enumerate}[label=\arabic*),leftmargin=*]
	\item $a-(k\nu)^2>0$, which is obviously the only case where \eqref{eq:affine_vort_Dir_spec} can occur: We have
	\[\beta^{-(k\nu)^2,\lambda}(y)=\frac{\sin(\sqrt{a-(k\nu)^2}(y+h))}{\sin(\sqrt{a-(k\nu)^2}h)},\]
	provided \eqref{eq:affine_vort_Dir_spec} fails to hold, and thus
	\[d(-(k\nu)^2,\lambda)=\sqrt{a-(k\nu)^2}\cot(\sqrt{a-(k\nu)^2}h)-\sigma\lambda^{-2}(k\nu)^2+\lambda^{-1}b-\lambda^{-2}g.\]
	\item $a-(k\nu)^2=0$: Here we have
	\[\beta^{-(k\nu)^2,\lambda}(y)=\frac{y+h}{h}\]
	and thus
	\[d(-(k\nu)^2,\lambda)=h^{-1}-\sigma\lambda^{-2}(k\nu)^2+\lambda^{-1}b-\lambda^{-2}g.\]
	\item $a-(k\nu)^2<0$: Here we have
	\[\beta^{-(k\nu)^2,\lambda}(y)=\frac{\sinh(\sqrt{(k\nu)^2-a}(y+h))}{\sinh(\sqrt{(k\nu)^2-a}h)}\]
	and thus
	\begin{align}\label{eq:affine_disp_rel_3}
		d(-(k\nu)^2,\lambda)=\sqrt{(k\nu)^2-a}\coth(\sqrt{(k\nu)^2-a}h)-\sigma\lambda^{-2}(k\nu)^2+\lambda^{-1}b-\lambda^{-2}g.
	\end{align}
\end{enumerate}
In cases 2) and 3), it is furthermore easy to see that automatically $d_\lambda(-(k\nu)^2,\lambda)\ne 0$ provided $d(-(k\nu)^2,\lambda)=0$, since, as in the case of constant vorticity, the (of $\lambda$ independent) discriminant $b^2+4(\sigma(k\nu)^2+g)\beta_y^{-(k\nu)^2,\lambda}(0)$ is strictly positive.

\paragraph{The case $a<0$}
Let us now further investigate the case $a<0$, which has the advantage that we are in situation 3) for any $k$, and extend the ideas of \cite{Martin13} to affine vorticity. We can solve \eqref{eq:affine_disp_rel_3} for $\lambda$ to obtain
\begin{align}\label{eq:aff_vort_lambda(k)}
	\lambda=\lambda^\pm(l)\coloneqq\frac{-b\pm\sqrt{b^2+4(\sigma l^2+g)\sqrt{l^2-a}\coth(\sqrt{l^2-a}h)}}{2\sqrt{l^2-a}\coth(\sqrt{l^2-a}h)}
\end{align}
with $l=k\nu$. In the following, we treat the \enquote{$+$} and \enquote{$-$} situation simultaneously; everything is to be read separately for both cases. Viewing $\lambda^\pm$ as a (smooth) function on $\R$, it is clear that $\lambda^\pm\gtrless 0$ and $\lambda^\pm$ is even; in particular,
\begin{align}\label{eq:affine_lambda_der_0}
	\lambda^\pm_l(0)=\lambda^\pm_{lll}(0)=0.
\end{align} 
Now let us denote $f(l)\coloneqq\sqrt{l^2-a}\coth(\sqrt{l^2-a}h)$. Differentiating the identity
\[\lambda^\pm(l)^2d(-(k\nu)^2,\lambda^\pm(l))=0\]
with respect to $l$ yields
\begin{align}
	\lambda^\pm_l(2\lambda^\pm f+b)&=-(\lambda^\pm)^2f_l+2\sigma l,\nonumber\\
	\lambda^\pm_{ll}(2\lambda^\pm f+b)&=-4\lambda^\pm\lambda^\pm_lf_l-2(\lambda^\pm_l)^2f-(\lambda^\pm)^2f_{ll}+2\sigma.\label{eq:affine_lambda_kk}
\end{align}
Thus, if $\lambda^\pm_l=0$ at some $l>0$, then
\begin{align*}
	\lambda^\pm_{ll}(2\lambda^\pm f+b)=2\sigma-(\lambda^\pm)^2f_{ll}=2\sigma\frac{f_l-lf_{ll}}{f_l},
\end{align*}
noticing that $f_l>0$ for $l>0$. A direct computation shows that
\[f_l-lf_{ll}=(lh)^3\tilde g(\sqrt{l^2-a}h)\]
with
\[\tilde g(x)\coloneqq\frac{\tilde g_1(x)}{\tilde g_2(x)}\coloneqq\frac{\sinh^2x\cosh x+x\sinh x-2x^2\cosh x}{x^3\sinh^3x}.\]
We have $\tilde g_2>0$ on $(0,\infty)$, $\tilde g_1(0)=\tilde g_1'(0)=\tilde g_1''(0)=0$, and
\[\tilde g_1'''(x)=\sinh x(7\sinh^2x-2x^2)+9\sinh x(\cosh^2x-1)+11\cosh x(\sinh x\cosh x-x)>0\]
for $x>0$ because of $\sinh x>x$, $\cosh x>1$. Therefore, it holds that $\tilde g>0$ on $(0,\infty)$. Moreover, notice that
\begin{align}\label{eq:affine_bracket_sign}
	2\lambda^\pm f+b=\frac{\sigma l^2+g+(\lambda^\pm)^2f}{\lambda^\pm}\gtrless 0.
\end{align}
Thus, putting everything together, $\lambda^\pm_{ll}\gtrless 0$ provided $\lambda^\pm_l=0$. In particular, $\lambda^\pm$ can have at most one critical point on $(0,\infty)$, which, if it exists, has to be a local minimum (maximum). Since moreover obviously $\lambda^\pm$ tends to $\pm\infty$ as $l\to\infty$, we conclude that the monotonicity properties of $\lambda^\pm$ can be characterised by its behaviour near $0$.

Let us therefore take a look at $l=0$. By \eqref{eq:affine_lambda_der_0}, \eqref{eq:affine_lambda_kk}, and \eqref{eq:affine_bracket_sign} we see that $\lambda^\pm_{ll}(0)$ has the same sign as $\pm(2\sigma-(\lambda^\pm(0))^2f_{ll}(0))$, or vanishes if and only if $2\sigma=(\lambda^\pm(0))^2f_{ll}(0)$. Let us first consider further the latter case. Differentiating \eqref{eq:affine_lambda_kk} twice more, evaluating at $0$, and using \eqref{eq:affine_lambda_der_0} and $\lambda^\pm_{ll}(0)=0$ yields
\[\lambda^\pm_{llll}(0)(2\lambda^\pm(0)f(0)+b)=-\lambda^\pm(0)^2f_{llll}(0)=3\lambda^\pm(0)^2h^3\tilde g(\sqrt{-a}h)>0\]
after a direct computation. In particular, $\lambda^\pm_{llll}(0)\gtrless 0$. To summarise, we have therefore proved:

If $2\sigma\ge(\lambda^\pm(0))^2f_{ll}(0)$, then $\lambda^\pm$ is strictly increasing (decreasing) on $[0,\infty)$. In particular, for every $k_0\in\N$, $F_{(w,\phi)}(\lambda^\pm(k_0\nu),0,0)$ has a one-dimensional kernel and thus the assumptions of Theorems \ref{thm:LocalBifurcation} and \ref{thm:GlobalBifurcation} are satisfied for $(k_0,\lambda)=(k_0,\lambda^\pm(k_0\nu))$.

If $2\sigma<(\lambda^\pm(0))^2f_{ll}(0)$, then there exists $l^\pm>0$ such that $\lambda^\pm$ is strictly decreasing (increasing) on $[0,l^\pm)$ and strictly increasing (decreasing) on $(l^\pm,\infty)$. In particular, such a conclusion as above cannot always be made. However, denoting $l^\pm_1$ the unique value of $l>0$ satisfying $\lambda^\pm(l)=\lambda^\pm(0)$ and recalling that $\lambda^\pm\to\pm\infty$ as $l\to\infty$, this conclusion can be made provided:
\begin{itemize}
	\item $k_0\nu\ge l^\pm_1$ (in particular, for sufficiently small wavelengths), or
	\item $k_0\nu<l^\pm_1$ and $k_0^\pm\notin\N$, where $k_0^\pm$ is the unique number $k>0$, $k\ne k_0$ such that $\lambda^\pm(k\nu)=\lambda^\pm(k_0\nu)$.
\end{itemize}
If, however, $k_0^\pm\in\N$ in the latter case, the corresponding kernel is two-dimensional, and one can follow the procedure described in Remark \ref{rem:two_sol} for $k_0$ and $k_0^\pm$.

Finally, we mention that $(\lambda^\pm(0))^2f_{ll}(0)$ depends only on $a$, $b$, $g$, and $h$, so that the above cases can be viewed as \enquote{coefficient of surface tension larger/smaller than a certain threshold}. Moreover, it is not difficult to see that
\[\lim_{a\to 0}(\lambda^\pm(0))^2f_{ll}(0)=\frac{b^2h^3}{3}+\frac{2gh^2}{3}\mp\frac{bh^2}{3}\sqrt{b^2h^2+4gh},\]
that is, the threshold in the case $a=0$ is recovered in the limit $a\to 0$, as one would expect.

\paragraph{The case $a>0$}
The case $a>0$ is more involved since $d(\cdot,\lambda)$ has poles at the values of $\mu$ for which \eqref{eq:affine_vort_Dir_spec} holds with $-(k\nu)^2$ replaced by $\mu$. In order to assure \eqref{ass:SL-spectrum}, we assume that
\begin{align}\label{eq:affine_a>0_ass}
	a\notin\left\{\frac{\pi^2}{h^2}m^2:m\in\N\right\}
\end{align}
in view of \eqref{eq:affine_vort_Dir_spec}. Without the presence of surface tension, a detailed analysis of the dispersion relation and the possibility of multi-dimensional kernels can be found in \cite{AasVar18,EhrnEschWahl11,EhrnWahl15}. However, taking surface tension into account, the situation becomes much more difficult due to the competing monotonicity behaviour of $\beta_y^{-(k\nu)^2,\lambda}(0)$, which is strictly increasing in $k$ (where it is defined), and $-\sigma\lambda^{-2}(k\nu)^2$, which is strictly decreasing in $k$, and, moreover, due to the fact that $d(-(k\nu)^2,\lambda)$ cannot be written as a function of only $k$ plus a function of only $\lambda$. In \cite[Lem. 3.6]{AasVar18} it was proved that the limit points of the set of all $a>0$ for which there exists $\lambda\in\R$ such that the corresponding kernel is at least two-dimensional are contained in $\{n^2+\frac{\pi^2}{h^2}m^2:n\in\N_0,m\in\N\}$.
A similar, yet much weaker result (the possible exceptional \enquote{small} set may now depend on $\lambda$) can also be proved in the presence of surface tension. In order to clarify the dependence of $d$ on $a$, we will write $d(-(k\nu)^2,\lambda,a)$ instead of $d(-(k\nu)^2,\lambda)$ in the following lemma.
\begin{lemma}\label{lma:affine_a>0}
	Let $\lambda\ne 0$ and denote
	\[J_\lambda\coloneqq\{a>0:\exists k_1,k_2\in\N,k_1\ne k_2\mathrm{\ with\ }d(-(k_1\nu)^2,\lambda,a)=d(-(k_2\nu)^2,\lambda,a)=0\}.\]
	Then the limit points of $J_\lambda$ are contained in
	\[A\coloneqq\left\{n^2+\frac{\pi^2}{h^2}m^2:n\in\N,m\in\N\right\}.\]
	In particular, $J_\lambda$ consists of isolated points, except possibly those that lie in $A$, and has countable closure.
\end{lemma}
\begin{proof}
	Notice that $a\notin J_\lambda$ if and only if $d(-(k\nu)^2,\lambda)$ is finite for all $k\in\N$. Now suppose that the assertion is false, that is, there exists $a\notin A$ and sequences $(a_i)\subset A$ with $a_i\neq a$ for all $i\in\N$ and $(k_{1,i}),(k_{2,i})\subset\N$ with $k_{1,i}<k_{2,i}$ for all $i\in\N$ such that $d(-(k_{1,i}\nu)^2,\lambda,a_i)=d(-(k_{2,i}\nu)^2,\lambda,a_i)$ for all $i\in\N$ and $a_i\to a$ as $i\to\infty$. Thus, in particular, $l_\lambda(k_{1,i},a_i)=l_\lambda(k_{2,i},a_i)$ where
	\begin{align*}
		l_\lambda(k,a)&\coloneqq\beta^{-(k\nu)^2,\lambda}_y(0)-\sigma\lambda^{-2}(k\nu)^2\\
		&=\begin{cases}\sqrt{a-(k\nu)^2}\cot(\sqrt{a-(k\nu)^2}h)-\sigma\lambda^{-2}(k\nu)^2\,&a>(k\nu)^2,\\h^{-1}-\sigma\lambda^{-2}(k\nu)^2,&a=(k\nu)^2,\\\sqrt{(k\nu)^2-a}\coth(\sqrt{(k\nu)^2-a}h)-\sigma\lambda^{-2}(k\nu)^2,&a<(k\nu)^2.\end{cases}
	\end{align*}
	Let us now denote $f_\lambda(x)\coloneqq x\coth(hx)-\sigma\lambda^{-2}x^2$, $x>0$. Then for large values of $x$, say, $x\ge r_\lambda$, the function $f_\lambda$ is strictly decreasing since $x\coth(hx)$ is exponentially close to $x$. Therefore, necessarily $\sqrt{(k_{1,i}\nu)^2-a_i}\le r_\lambda$ if $(k_{1,i}\nu)^2>a_i$. In particular, $(k_{1,i})$ is bounded. Passing to a suitable subsequence, we can thus assume that $k_{1,i}=k_1$ is constant. Since $f_\lambda(x)\to-\infty$ as $x\to\infty$ and $l_\lambda(k_{2,i},a_i)=l_\lambda(k_1,a_i)$, we conclude that also $(k_{2,i})$ is bounded. Passing to another suitable subsequence, we find that $l_\lambda(k_1,a_i)=l_\lambda(k_2,a_i)$ for some $k_1<k_2$. Since $a\notin A$, the function $l_\lambda(k_1,\cdot)-l_\lambda(k_2,\cdot)$ is finite and analytic in a neighbourhood of $a$, with $a$ being a limit point of its zero set. Thus, this function has to vanish in a neighbourhood of $a$, which obviously does not hold true.
\end{proof}
Let us now consider the transversality condition: At a solution of $d(-(k\nu)^2,\lambda)=0$, it is clear that the transversality condition is satisfied if and only if the discriminant $b^2+4(\sigma(k\nu)^2+g)\beta_y^{-(k\nu)^2,\lambda}(0)$ does not vanish, and this can only happen if $(k\nu)^2<a$, as was already observed earlier. On the other hand, we can also write the transversality condition as
\[d_\lambda(-(k\nu)^2,\lambda)=2\sigma\lambda^{-3}(k\nu)^2-\lambda^{-2}b+2\lambda^{-3}g\ne 0\quad\Leftrightarrow\quad 2(\sigma(k\nu)^2+g)\ne\lambda b.\]

We can now combine our results in the case $a>0$ as follows.
\begin{proposition}
	Let $\lambda\ne 0$ and suppose that the hypotheses of Theorems \ref{thm:LocalBifurcation} and \ref{thm:GlobalBifurcation}, for this fixed $\lambda$, are violated. Then one of the following alternatives occurs:
	\begin{enumerate}[label=(\roman*)]
		\item the dispersion relation $d(-(k\nu)^2,\lambda)=0$ has no solution $k\in\N$ at all;
		\item $a\in N_\lambda$, with $N_\lambda$ being defined as the union of $J_\lambda$ and the set appearing in \eqref{eq:affine_a>0_ass}; moreover, all the conclusions made on $J_\lambda$ in Lemma \ref{lma:affine_a>0} obviously also hold for $N_\lambda$;
		\item there exists exactly one $k\in\N$ such that $d(-(k\nu)^2,\lambda)=0$, and, moreover, $(k\nu)^2<a$ and
		\begin{align*}
			b^2+4(\sigma(k\nu)^2+g)\sqrt{a-(k\nu)^2}\cot(\sqrt{a-(k\nu)^2}h)&=0\quad\mathrm{or,\ equivalently,}\\
			2(\sigma(k\nu)^2+g)&=\lambda b.
		\end{align*}
	\end{enumerate}
\end{proposition}
\begin{remark}
	\begin{itemize}
		\item It is, for example, clear that alternative (iii) cannot occur in the case of a purely linear vorticity function, that is, $b=0$.
		\item A concrete example in the case $a>0$ for which all hypotheses of Theorems \ref{thm:LocalBifurcation} and \ref{thm:GlobalBifurcation} are satisfied is the following: Let $0<a<\nu$ such that $a\ne\pi^2m^2/h^2$ for all $m\in\N$. Thus, \eqref{eq:affine_vort_Dir_spec} does not hold for any $k\in\N_0$ and only case 3) occurs for $k\in\N$. Moreover, the dispersion relation can be solved as in \eqref{eq:aff_vort_lambda(k)} before. Since $\lambda^\pm$, now viewed as a function of $l\in[\nu,\infty)$, converges to $\pm\infty$ and is strictly monotone for large $k$, we can choose a sufficiently large $k_0\in\N$ together with $\lambda=\lambda^+(k_0\nu)$ or $\lambda=\lambda^-(k_0\nu)$ to ensure that the corresponding kernel is exactly one-dimensional and the transversality condition holds.
	\end{itemize}
\end{remark}

\section{Global bifurcation}\label{sec:GlobalBifurcation}
The theory for local bifurcation having set up, we now turn to global bifurcation, which is of course the main motivation of our formulation \enquote{identity $+$ compact}. To this end, we first state the global bifurcation theorem by Rabinowitz.
\begin{theorem}\label{thm:Rabinowitz}
	Let $X$ be a Banach space, $U\subset\R\times X$ open, and $F\in C(U;X)$. Assume that $F$ admits the form $F(\lambda,x)=x+f(\lambda,x)$ with $f$ compact, and that $F_x(\cdot,0)\in C(\R;L(X,X))$. Moreover, suppose that $F(\lambda_0,0)=0$ and that $F_x(\lambda,0)$ has an odd crossing number at $\lambda=\lambda_0$. Let $S$ denote the closure of the set of nontrivial solutions of $F(\lambda,x)=0$ in $\R\times X$ and $\mathcal C$ denote the connected component of $S$ to which $(\lambda_0,0)$ belongs. Then one of the following alternatives occurs:
	\begin{enumerate}[label=(\roman*)]
		\item $\mathcal C$ is unbounded;
		\item $\mathcal C$ contains a point $(\lambda_1,0)$ with $\lambda_1\neq\lambda_0$;
		\item $\mathcal C$ contains a point on the boundary of $U$.
	\end{enumerate}
\end{theorem}
The proof of this theorem in the case $U=X$ can be found in \cite[Thm. II.3.3]{Kielhoefer} and is practically identical to the proof for general $U$.

Now we can prove the following result.

\begin{theorem}\label{thm:GlobalBifurcation}
	Assume that there exists $\lambda_0\neq 0$ such that \eqref{ass:SL-spectrum} holds for $\lambda=\lambda_0$ and such that the dispersion relation
	\[d(-(k\nu)^2,\lambda_0)=0,\]
	with $d$ given by \eqref{eq:d(k,lambda)}, has exactly one solution $k_0\in\N$, and assume that the transversality condition
	\[d_\lambda(-(k_0\nu)^2,\lambda_0)\neq 0\]
	holds. Let $S$ denote the closure of the set of nontrivial solutions of $F(\lambda,w,\phi)=0$ in $\R\times X$ satisfying \eqref{eq:additional_requirements} (that is, solutions giving rise to a proper water wave) and $\mathcal C$ denote the connected component of $S$ to which $(\lambda_0,0,0)$ belongs. Then one of the following alternatives occurs:
	\begin{enumerate}[label=(\roman*)]
		\item $\mathcal C$ is unbounded in the sense that there exists a sequence $(\lambda_n,w_n,\phi_n)\in\mathcal C$ such that
		\begin{enumerate}[label=(\alph*)]
			\item $|\lambda_n|\to\infty$, or
			\item $\|w_n\|_{C_\per^{0,\delta}(\R)}\to\infty$ for any $\delta\in(6/7,1]$, or
			\item $\|\gamma((\phi_n+\psi^{\lambda_n})\circ H[w_n+h]^{-1})\|_{L^p(\Omega_{w_n}^*)}\to\infty$ for any $p>1$ (that is, the total vorticity in each copy of the fluid domain measured in $L^p$ is unbounded)
		\end{enumerate}
		as $n\to\infty$;
		\item $\mathcal C$ contains a point $(\lambda_1,0,0)$ with $\lambda_1\neq\lambda_0$;
		\item $\mathcal C$ contains a sequence $(\lambda_n,w_n,\phi_n)$ such that $(w_n)$ converges to some $w$ in the space $C_{0,\per,\e}^{1,\alpha}(\R)$ and the surface $S_w$ determined by $w$ via \eqref{eq:Sw} is \emph{not} of class $C^{1,\beta}$ for any $\beta>0$;
		\item $\mathcal C$ contains a point $(\lambda,w,\phi)$ such that
		\[x\mapsto(x+(\Ch w)(x),w(x)+h)\mathrm{\ is\ \textit{not}\ injective\ on\ }\R,\]
		that is, self-intersection of the surface profile occurs;
		\item $\mathcal C$ contains a point $(\lambda,w,\phi)$ such that there exists $x\in\R$ with
		\[w(x)=-h,\]
		that is, intersection of the surface profile with the flat bed occurs.
	\end{enumerate}
\end{theorem}
\begin{proof}
	As was already observed in Lemma \ref{lma:M_prop}, our nonlinear operator $F$ is of class $C^2$ and admits the form \enquote{identity $+$ compact} on each $\Om_\varepsilon$, $\varepsilon>0$. Moreover, it is well-known that $F_{(w,\phi)}(\lambda,w,\phi)$ has an odd crossing number at $(\lambda_0,0,0)$ provided $F_{(w,\phi)}(\lambda_0,0,0)$ is a Fredholm operator with index zero and one-dimensional kernel, and the transversality condition holds. These properties, in turn, are consequences of the hypotheses of the theorem in view of Lemmas \ref{lma:kernel} and \ref{lma:transversality_condition} since $F_{(w,\phi)}(\lambda_0,0,0)$ coincides with $\LL(\lambda_0)$ up to an isomorphism. For each $\varepsilon>0$, we can thus apply Theorem \ref{thm:Rabinowitz} with $U$ chosen to be the interior of $\Om_\varepsilon$. Suppose now that neither alternative (iv) nor (v) is valid. Then, on each $\Om_\varepsilon$, $\mathcal C$ coincides with its counterpart obtained from Theorem \ref{thm:Rabinowitz}. Since $\varepsilon>0$ is arbitrary and $\Om=\bigcup_{\varepsilon>0}\Om_\varepsilon$, it is evident that $\mathcal C$ has to contain a point in $\partial\Om$ whenever $\mathcal C$ is bounded in $\R\times X$ and (ii) fails to hold.
	
	Let us further consider the alternative that $\mathcal C$ contains a point in $\partial\Om$. Then, as stated in (iii), there exists a sequence $(\lambda_n,w_n,\phi_n)$ such that $(w_n)$ converges to some $w$ in $C_{0,\per,\e}^{1,\alpha}(\R)$ and 
	\begin{align}\label{eq:alternative(iii)}
	1+(\Ch w')(x)=w'(x)=0\text{ for some }x\in\R.
	\end{align}
	(We now essentially follow the proof of the reverse direction of \cite[Thm. 2.2]{ConsVarva11}.) Then $H[w+h]=U[w+h]+iV[w+h]$ is holomorphic, and, since $w\in C_\per^{1,\alpha}(\R)$, moreover $U[w+h],V[w+h]\in C^{1,\alpha}(\overline{\Omega_h})$ by \cite[Lem. 2.1]{ConsVarva11}. Additionally assuming that \eqref{eq:additional_requirements} holds for $w$ (otherwise we (iv) or (v) would be valid), an application of the Darboux--Picard Theorem shows that $H[w+h]$ is a conformal mapping from $\Omega_h$ onto $\Omega_w$, which admits an extension as a homeomorphism between the closures of	these domains, with $\R\times\{0\}$ being mapped onto $S_w$ and $\R\times\{-h\}$ being mapped onto $\R\times\{0\}$. Since this conformal mapping is unique up to translations in the variable $x$ (which of course leave regularity properties invariant) by Lemma \ref{lma:ConformalMapping}(ii), the surface cannot be of class $C^{1,\beta}$ for any $\beta>0$ in view of \eqref{eq:alternative(iii)} and Lemma \ref{lma:ConformalMapping}(iv) combined with \eqref{eq:length_of_nablaV_on_top}.
	
	It remains to investigate alternative (i) further. In order to show that it can be stated as above, we have to show that, in view of alternative (i) of Theorem \ref{thm:Rabinowitz}, $\mathcal C$ is bounded in $\R\times X$ if (a)--(c) (and also the other alternatives (ii),(iii)) fail to hold. In particular, we shall assume that $(\lambda,w,\phi)\in\mathcal C$, the $w$'s are bounded in $C_\per^{0,\delta}(\R)$ for some $\delta\in(6/7,1)$, $\|\gamma((\phi+\psi^{\lambda})\circ H[w+h]^{-1})\|_{L^p(\Omega_w^*)}$ remains bounded for some $p>1$, and $\K(w)\ge\varepsilon$ for some $\varepsilon>0$. Since the  $L^\infty$-norms of the $w$'s are bounded, the Lebesgue measure of $\Omega_w^*$ is bounded so that we can assume without loss of generality that $p$ is close enough to $1$ such that $\eta\coloneqq2-2/p<\min\{1/(1-\delta)-7,\alpha\}$. We now consider $\delta$, $\varepsilon$, and $\eta$ fixed, and let $\beta\in(0,\eta]$, $r,t\in(1,\infty)$ to be chosen later. In the following, $c>0$ denotes some generic constants that may depend on $\delta$, $\varepsilon$, $\eta$, $\beta$, $r$, and $t$, and can change from step to step, and we will frequently make use of Young's inequality $|ab|\le a^q/q+b^{q'}/q'$ for $a,b\in\R$ and $q\in(1,\infty)$, where $1/q+1/q'=1$; we let $r'$ and $t'$ such that $1/r+1/r'=1$ and $1/t+1/t'=1$. First note that
	\[|\gamma(s)|\le c|s|+c,\quad s\in\R,\]
	and
	\[\|V[w+h]\|_{C_\per^{1,\beta}(\overline{\Omega_h})}\le c\|w\|_{C_\per^{1,\beta}(\R)}+c\]
	similar to \eqref{eq:est_V}. Now, along $\mathcal C$ we have $\phi=\A(\lambda,w,\phi)$ and thus
	\begin{align*}
		\|\phi\|_{C_\per^{0,\eta}(\overline{\Omega_h})\cap H_\per^1(\Omega_h)}&\le c\|\phi\|_{W^{2,p}(\Omega_h^*)}\le c\|-\gamma(\phi+\psi^\lambda)|\nabla V[w+h]|^2+\gamma(\psi^\lambda)\|_{L^p(\Omega_h^*)}\\
		&\le c\|\gamma(\phi+\psi^\lambda)|\nabla V[w+h]|^{1/p}\|_{L^p(\Omega_h^*)}\|\nabla V[w+h]\|_\infty^{1+\eta/2}+c|\lambda|+c\\
		&\le c\|\gamma((\phi+\psi^{\lambda})\circ H[w+h]^{-1})\|_{L^p(\Omega_w^*)}\left(\|w\|_{C_\per^{1,\beta}(\R)}^{1+\eta/2}+1\right)+c|\lambda|+c\\
		&\le c\|\gamma((\phi+\psi^{\lambda})\circ H[w+h]^{-1})\|_{L^p(\Omega_w^*)}^{r'}+c\|w\|_{C_\per^{1,\beta}(\R)}^{r(1+\eta/2)}+c|\lambda|+c
	\end{align*}
	after using Sobolev's embedding, the Calderón-Zygmund inequality (see \cite[Chapter 9]{GilbargTrudinger}; notice that on the right-hand side the term $\|\phi\|_{L^p(\Omega_h^*)}$ can be left out because of unique solvability of the Dirichlet problem associated to $\Delta$), $2-1/p=1+\eta/2$, the Schauder estimate for $V[w+h]$, and a change of variables via $H[w+h]$. Therefore, similar to \eqref{eq:est_R}, \eqref{eq:est_M1} (but a bit refined), we infer the Schauder estimates 
	\begin{align}\label{eq:est_phi}
		&\|\phi\|_{C_\per^{2,\beta}(\overline{\Omega_h})}\le c(\|\phi\|_{C_\per^{0,\beta}(\overline{\Omega_h})}+|\lambda|+1)\|V[w+h]\|_{C_\per^{1,\beta}(\overline{\Omega_h})}^2+c|\lambda|+c\nonumber\\
		&\le c\|\phi\|_{C_\per^{0,\beta}(\overline{\Omega_h})}^{\frac{2}{r(1+\eta/2)}+1}+c\|V[w+h]\|_{C_\per^{1,\beta}(\overline{\Omega_h})}^{2+r(1+\eta/2)}+c|\lambda|^{\frac{2}{r(1+\eta/2)}+1}+c\nonumber\\
		&\le c\|\gamma((\phi+\psi^{\lambda})\circ H[w+h]^{-1})\|_{L^p(\Omega_w^*)}^c+c\|w\|_{C_\per^{1,\beta}(\R)}^{2+r(1+\eta/2)}+c|\lambda|^c+c,
	\end{align}
	then, using the global Lipschitz continuity of $|\cdot|$ on $\R^2$ and $s\mapsto 1/s$ on $[\varepsilon,\infty)$ for the terms involving $\K(w)$,
	\begin{align*}
		&\|R(\lambda,w,\phi)\|_{C_\per^{0,\beta}(\R)}\\
		&\le c(\|w\|_{C_\per^{1,\beta}(\R)}+1)\left(\|\phi\|_{C_\per^{1,\beta}(\overline{\Omega_h})}^2+\lambda^2+\|w\|_{C_\per^{0,\beta}(\R)}\right)\\
		&\le c\|w\|_{C_\per^{1,\beta}(\R)}^{5+r(2+\eta)}+c\|\phi\|_{C_\per^{1,\beta}(\overline{\Omega_h})}^{\frac{2}{4+r(2+\eta)}+2}+c|\lambda|^{\frac{2}{4+r(2+\eta)}+2}+c\\
		&\le c\|\gamma((\phi+\psi^{\lambda})\circ H[w+h]^{-1})\|_{L^p(\Omega_w^*)}^c+c\|w\|_{C_\per^{1,\beta}(\R)}^{5+r(2+\eta)}+c|\lambda|^c+c,
	\end{align*}
	and finally
	\begin{align*}
		&\|\M_1(\lambda,w,\phi)\|_{C_\per^{2,\beta}(\R)}\le c(1+\|w\|_{C_\per^{1,\beta}(\R)})\|R(\lambda,w,\phi)\|_{C_\per^{0,\beta}(\R)}\\
		&\le c\|w\|_{C_\per^{1,\beta}(\R)}^{6+r(2+\eta)}+c\|R(\lambda,w,\phi)\|_{C_\per^{0,\beta}(\R)}^{\frac{1}{5+r(2+\eta)}+1}+c\\
		&\le c\|\gamma((\phi+\psi^{\lambda})\circ H[w+h]^{-1})\|_{L^p(\Omega_w^*)}^c+c\|w\|_{C_\per^{1,\beta}(\R)}^{6+r(2+\eta)}+c|\lambda|^c+c.
	\end{align*}
	Therefore and since we have $w=\M_1(\lambda,w,\phi)$ along $\mathcal C$,
	\begin{align*}
		\|w\|_{C_\per^{2,\beta}(\R)}\le c\|\gamma((\phi+\psi^{\lambda})\circ H[w+h]^{-1})\|_{L^p(\Omega_w^*)}^c+c\|w\|_{C_\per^{1,\beta}(\R)}^{6+r(2+\eta)}+c|\lambda|^c+c.
	\end{align*}
	By interpolation we have
	\[c\|w\|_{C_\per^{1,\beta}(\R)}^{6+r(2+\eta)}\le c\|w\|_{C_\per^{2,\beta}(\R)}^{(6+r(2+\eta))\frac{1+\beta-\delta}{2+\beta-\delta}}\|w\|_{C_\per^{0,\delta}(\R)}^{\frac{6+r(2+\eta)}{2+\beta-\delta}}\le\frac12\|w\|_{C_\per^{2,\beta}(\R)}^{t(6+r(2+\eta))\frac{1+\beta-\delta}{2+\beta-\delta}}+c\|w\|_{C_\per^{0,\delta}(\R)}^{t'\frac{6+r(2+\eta)}{2+\beta-\delta}}\]
	due to Young's inequality with epsilon. We now choose $\beta$ close enough to $0$ and $r,t$ close enough to $1$ such that
	\[t(6+r(2+\eta))\frac{1+\beta-\delta}{2+\beta-\delta}=1.\]
	This is possible since
	\[(6+(2+\eta))\frac{1-\delta}{2-\delta}<1\quad\Leftrightarrow\quad\eta<\frac{1}{1-\delta}-7.\]
	Hence,
	\begin{align*}
		\frac12\|w\|_{C_\per^{2,\beta}(\R)}\le c\|\gamma((\phi+\psi^{\lambda})\circ H[w+h]^{-1})\|_{L^p(\Omega_w^*)}^c+c\|w\|_{C_\per^{0,\delta}(\R)}^{t'\frac{6+r(2+\eta)}{2+\beta-\delta}}+c|\lambda|^c+c.
	\end{align*}
	This and inserting this inequality in \eqref{eq:est_phi} completes the proof in view of the obvious estimate $\|(w,\phi)\|_X\le c(\|w\|_{C_\per^{2,\beta}(\R)}+\|\phi\|_{C_\per^{2,\beta}(\overline{\Omega_h})})$.
\end{proof}
\begin{remark}
	If, for example, $\gamma$ is bounded, the theorem remains obviously valid if alternative (i)(c) is left out. 
\end{remark}
\begin{remark}
	In the case of pure gravity waves, it is sometimes possible to get rid of an alternative like (ii) above; cf. \cite{ConsStr04,ConsStrVarv16}, for example. The argument to eliminate such an outcome is typically based on a maximum principle argument, and this type of argument appears to be unavailable when capillary effects are taken into account.
\end{remark}

\begin{proposition}
	In Theorem \ref{thm:GlobalBifurcation} the alternative (iii) can be replaced by
	\begin{enumerate}
		\item[(iii')]
		\begin{enumerate}[label=(\alph*)]
			\item $\inf_{(\lambda,w,\phi)\in\mathcal C}\min_{\R}\SL|\nabla(\phi+\psi^\lambda)|=0$, or
			\item $\sup_{(\lambda,w,\phi)\in\mathcal C}\max_{\R}\kappa[w]=\infty$ (that is, the maximal mean curvature of the surface profile is unbounded from above), or
			\item $\sup_{(\lambda,w,\phi)\in\mathcal C}Q(\lambda,w,\phi)=\infty$ (that is, the free energy at the surface is unbounded from above).
		\end{enumerate}
	\end{enumerate}
\end{proposition}
\begin{proof}
	Recalling the original Bernoulli equation (cf. \eqref{eq:OriginalBernoulliFlat}, \eqref{eq:length_of_nablaV_on_top}) we have
	\[\frac{\SL|\nabla(\phi+\psi^\lambda)|^2}{2((1+\Ch w')^2+w'^2)}-\sigma\kappa[w]+gw=Q\]
	for $(\lambda,w,\phi)\in\mathcal C$. Thus, if (iii) holds, that is, $\inf_{(\lambda,w,\phi)\in\mathcal C}\min_{\R}(1+\Ch w')^2+w'^2=0$ in particular, then (iii')(a), (iii')(b), or (iii')(c) occurs or $w$ is unbounded from below, the latter of which, however, is already absorbed in alternative (v) (or, in other words, is prevented by the bed).
\end{proof}

\paragraph{Acknowledgement} This project has received funding from the European Research Council (ERC) under the European Union’s Horizon 2020 research and innovation programme (grant agreement no 678698).

\bibliographystyle{amsplain}
\bibliography{bib_2DBifurcation}

\end{document}